\documentclass[conference,letter,onecolumn,10pt]{IEEEtran}

\usepackage{amsmath, amsthm, amssymb}
\usepackage{latexsym}
\usepackage{graphicx}
\usepackage{verbatim}
\usepackage{mathrsfs}
\usepackage{float}
\usepackage[lofdepth, lotdepth]{subfig}
\usepackage{array}
\usepackage{url}
\usepackage{cite}
\usepackage[table]{xcolor}
\usepackage{enumerate}
\usepackage{eucal}
\usepackage{stmaryrd}
\usepackage{mathtools}
\usepackage{pgf}
\usepackage{tikz}
\usetikzlibrary{arrows,automata}
\usepackage[latin1]{inputenc}
\usetikzlibrary{automata,positioning}

\usepackage[margin=1in]{geometry}

\usepackage{algorithm}
\usepackage[noend]{algpseudocode}

\theoremstyle{plain}
\newtheorem{theorem}{Theorem}
\newtheorem{proposition}[theorem]{Proposition}
\newtheorem{lemma}[theorem]{Lemma}

\newtheorem{corollary}[theorem]{Corollary}
\newtheorem{construction}{Construction}
\theoremstyle{definition}
\newtheorem{definition}{Definition}
\newtheorem{example}{Example}
\newtheorem{remark}{Remark}

\newcommand{\C}{{\mathcal C}}
\newcommand{\D}{{\mathcal D}}

\newcommand{\X}{{\mathcal X}}

\newcommand{\bL}{{\boldsymbol L}}

\newcommand{\ba}{{\boldsymbol u}}
\newcommand{\bb}{{\boldsymbol v}}
\newcommand{\bp}{{\boldsymbol p}}
\newcommand{\bq}{{\boldsymbol q}}
\newcommand{\bu}{{\boldsymbol u}}
\newcommand{\bv}{{\boldsymbol v}}

\newcommand{\by}{{\boldsymbol y}}
\newcommand{\bt}{{\boldsymbol t}}

\newcommand{\bw}{{\boldsymbol{w}}}

\newcommand{\br}{{\boldsymbol{r}}}
\newcommand{\bx}{{\boldsymbol{x}}}
\newcommand{\bz}{{\boldsymbol{z}}}

\renewcommand{\ge}{\geqslant}
\renewcommand{\le}{\leqslant}

\newcommand{\floor}[1]{\left\lfloor #1\right\rfloor}

\begin{document}

\title{~\\[30mm]
Deciding the Confusability of Words under Tandem Repeats}

\author{~\\[2mm]
   \IEEEauthorblockN{
	Yeow Meng Chee,
	Johan Chrisnata,
	Han Mao Kiah, and
	Tuan Thanh Nguyen}

   \IEEEauthorblockA{
   School of Physical and Mathematical Sciences,
	Nanyang Technological University, Singapore\\
email: $\{${ymchee}, {jchrisnata}, {hmkiah}, {nguyentu001}$\}$@ntu.edu.sg 
 }
}
\maketitle
\pagestyle{plain}
\vspace{30mm}

\begin{abstract}
Tandem duplication in DNA is the process of inserting a copy of a segment of DNA adjacent to the original position.
Motivated by applications that store data in living organisms, 
Jain {\em et al.} (2016) proposed the study of codes that correct tandem duplications to improve the reliability of data storage. 
We investigate algorithms associated with the study of these codes.

Two words are said to be ${\le}k$-confusable if 
there exists two sequences of tandem duplications of lengths at most $k$
such that the resulting words are equal.
We demonstrate that the problem of deciding whether two words is ${\le}k$-confusable is linear-time solvable 
through a characterisation that can be checked efficiently for $k=3$.
Combining with previous results, the decision problem is linear-time solvable for $k\le 3$. 
We conjecture that this problem is undecidable for $k>3$.

Using insights gained from the algorithm, 
we study the size of tandem-duplication codes.
We improve the previous known upper bound and 
then construct codes with larger sizes as compared to the previous constructions.
We determine the sizes of optimal tandem-duplication codes for lengths up to {twenty},
develop recursive methods to construct tandem-duplication codes for all word lengths, and 
compute explicit lower bounds for the size of optimal tandem-duplication codes for lengths {from 21 to 30}.
%

\end{abstract}


\newpage

\section{Introduction}
At the beginning of the millenium, Lander {\em et al.} \cite{Lander} published a draft sequence of the human genome  
and reported that more than 50\% of the human genome consists of repeated substrings \cite{Lander}. 
There are two types of common repeats: {\em interspersed} repeats and {\em tandem} repeats. 
Interspersed repeats are caused by transposons 
when a segment of DNA is copied and pasted into new positions of the genome.
In contrast, tandem repeats are caused by slipped-strand mispairings \cite{Mundy}, and 
they occur 
when a pattern of one or more nucleotides is repeated and the repetitions are 
adjacent to each other. For example, consider the word {\tt AGTAGTCTGC}. The substring {\tt AGTAGT} is a tandem repeat, and we say that 
{\tt AGTAGTCTGC} is generated from {\tt AGTCTGC} by a {\em tandem duplication} of length three. 
Tandem repeats are believed to be the cause of several genetic disorders \cite{Usdin,Sutherland,Fondon}
and this motivated the study of tandem duplication mechanisms in a variety of contexts.

\begin{itemize}
\item \textbf{Formal languages}: 
Leopold {\em et al.} \cite{Leupold1}, \cite{Leupold2} 
defined the {\em unbounded duplication language} and the {\em $k$-bounded duplication language} to be the
set of words generated by seed word under 
tandem duplications of unbounded length and  tandem duplications of length up to $k$, respectively.
They investigated certain decidability problems involving unbounded duplication languages and 
showed that all $k$-bounded duplication languages are context free. 
Furthermore, they showed that $k$-bounded duplication language is always regular for any binary seed and $k \ge 1$. 
On the other hand, the $k$-bounded duplication language is not regular 
for any \textit{square-free} seed word over an alphabet of at least three letters and $k \ge 4$. 
More recently, Jain {\em et al.} \cite{Jain2} completed this characterization and 
proved that $k$-bounded duplication languages are regular for $k \le 3$.
\item \textbf{Information theory}: Farnoud {\em et al.} \cite{Farnoud} introduced the concept of {\em capacity} to 
determine average information content of a $k$-bounded duplication language. 
Later, Jain {\em et al.} \cite{Jain2} introduced the notion of {\em expressiveness} to
measure a language's capability to generate words that contain certain desired substrings.
A complete characterization of fully expressive $k$-bounded duplication languages was 
provided by Jain {\em et al.} for all alphabet sizes and all $k$.
\item \textbf{Codes correcting tandem duplications}: 
Motivated by applications that store data in living organisms \cite{Arita,Heider,Liss}, 
Jain {\em et al.} \cite{Jain} proposed the study of codes that correct tandem duplications to improve the reliability of data storage. 
They investigated various types of tandem duplications and provided optimal code construction in the case 
where duplication length is at most two.
\end{itemize}

In this paper, we study the last problem and investigate algorithms associated with these codes. 
In particular, given two words ${\bf{x}}$ and ${\bf{y}}$, we look for efficient algorithms that answer the following question:
 When are the words ${\bf{x}}$ and ${\bf{y}}$ \textit{confusable} under tandem repeats? 
 In other words, are there two sequences of tandem duplications such that the resulting words ${\bf{x'}}$ and ${\bf{y'}}$ are equal?

Interestingly, as we demonstrate, even for small duplication lengths, the solutions to this question are nontrivial. 
This is surprising as efficient algorithms are well known for analogous questions in other problems in
string matching \cite{Navarro:2001}.  
Before we give an account of these results, we introduce some necessary notations and provide a formal statement of our problem.

%
%

\section{Preliminaries}


Let $\Sigma_q=\{0,1,\cdots,q-1\}$ be an alphabet of $q\ge 2$ symbols.
For a positive integer $n$, let $\Sigma_q^n$ denote the set of all strings or words of length $n$ over $\Sigma_q$,
and let $\Sigma_q^*$ denote the set of all finite words over $\Sigma_q$, or
the {\em Kleene closure} of $\Sigma_q$.
Given two words ${\bx}, {\by} \in \Sigma_q^*$, we denote their concatenation by ${\bx\by}$. 

We state the {\em tandem duplication} rules. 
For {nonnegative} integers $k\le n$ and $i\le n-k$, we define $T_{i,k} : \Sigma_q^{n} \to \Sigma_q^{n+k}$ such that
\[   
T_{i,k}({\bx}) = \bu\bv\bv\bw, \mbox{~where~} \bx=\bu\bv\bw,\, |\bu|=i,\, |\bv|=k.
\]

If a finite sequence of tandem duplications of length $k$ is performed to obtain ${\by}$ from ${\bx}$, 
then we say that ${\by}$ is a $k$-\textit{descendant} of ${\bx}$, or 
${\bx}$ is a $k$-\textit{ancestor} of ${\by}$, and denote this relation by ${\bx} \xRightarrow[k]{*} {\by}$.
Formally, ${\bx} \xRightarrow[k]{*} {\by}$ means that for some positive $t$, there exist $t$ non-negative integers $i_1,i_2, \ldots i_t$ such that ${\by}= T_{i_t,k}\circ T_{i_{t-1},k}\circ \cdots \circ T_{i_1,k}({\bx})$.

Given a sequence $\bx$ and integer $k$, we consider the set of words that may be obtained from $\bx$
via a finite number of tandem duplications of length $k$. We define
the $k$-{\em descendant cone} of $\bx$ to be the set of all $k$-descendants of $\bx$
and denote this cone by $D_k^*(\bx)$. In other words, 
\[D_k^*({\bx})\triangleq \left\{{\by} \in \Sigma_q^* \mid {\bx} \xRightarrow[k]{*} \by \right\}.\]

Our work studies tandem duplications whose lengths are upper bounded by an integer $k$
and we extend the previous definitions. 
Formally, if a finite sequence of tandem duplications of length up to $k$ is performed to obtain ${\by}$ from ${\bx}$, 
then we say that ${\by}$ is a ${\le}k$-{\em descendant} of ${\bx}$ and 
denote this relation by ${\bx} \xRightarrow[\le k]{*} y$.
We have analogous definitions of ${\le}k$-{\em descendant}, ${\le}k$-{\em ancestor}, 
and  ${\le}k$-{\em descendant cone} $D_{\le k}^*(\bx)$.

\begin{example}
Consider ${\bx}=01210$ over $\Sigma_3$. 
Since $T_{1,3}({\bx}) = 01211210$ and $T_{0,2}(01211210)=0101211210$,
we write $01210 \xRightarrow[3]{*}  01211210$ and 
$01210 \xRightarrow[\le 3]{*}  0101211210$.
Alternatively, we have that 
 $01211210 \in D_3^*({\bx})$, and 
 $0101211210 \in D_{\le 3}^*({\bx})$. 
\end{example}

\subsection{Problem Formulation}
Motivated by applications that store data on living organisms, Jain {\em et al.} \cite{Jain} 
looked at the ${\le}k$-descendant cones of a pair of words and 
asked whether the two cones have a nontrivial intersection. 
Specifically, we introduce the notion of confusability.

\begin{definition}[Confusability] Two words ${\bx}$ and ${\by}$, are said to be $k$-\textit{confusable} if 
$D_{k}^*({\bx}) \cap D_{k}^*({\by}) \neq \varnothing.$ 
Similarly, they are said to be ${\le}k$-\textit{confusable} if 
$D_{\le k}^*({\bx}) \cap D_{\le k}^*({\by}) \neq \varnothing.$ 
\end{definition}

To design error-correcting codes that store information in the DNA of living organisms, 
Jain {\em et al.} then proposed the use of codewords that are not pairwise confusable.

\begin{definition}[${\le}k$-Tandem-Duplication Codes] 
A subset ${\cal C} \subseteq \Sigma_q^n$ is a ${\le}k$-{\em tandem-duplication code} if 
for all ${\bx},{\by} \in {\cal C}$ and ${\bx} \ne {\by}$, we have that ${\bx}$ and ${\by}$ are not ${\le}k$-confusable.
We say that ${\cal C}$ is an $(n,{\le}k;q)$-tandem-duplication code or $(n,{\le}k;q)$-TD code.
\end{definition}

Therefore, to determine if a set of words is a tandem-duplication code, 
we need to verify that all pairs of distinct words are not confusable.
Hence, we state our problem of interest.
\vspace{2mm}

\hspace{40mm}{\sc Confusability Problem}

\hspace{40mm}{\bf Instance}: Two words $\bx$ and $\by$ over $\Sigma_q$, and an integer $k$

\hspace{40mm}{\bf Question}: Are $\bx$ and $\by$ ${\le}k$-confusable?
\vspace{2mm}

While the confusability problem is a natural question, efficient algorithms are only known for the case where $k\in\{1,2\}$.
We review these results in the next subsection. 
%
%
%

\subsection{Previous Work}\label{sec:previous}

\noindent{\bf Confusability}.
We summarize known efficient algorithms that determine whether two words are $k$-confusable for all $k$, and
whether they are ${\le}k$-confusable for $k\in\{1,2\}$. 
We then highlight why these methods cannot be extended for the case $k\ge 3$.
Crucial to the algorithms is the concept of irreducible words and roots.

\begin{definition} A word ${\bx}$ is said to be $k$-\textit{irreducible} if ${\bx}$ cannot be deduplicated into shorter words with deduplication of length $k$. In other words, if $\by\xRightarrow[k]{*} \bx$, then $\by=\bx$. 
The set of $k$-irreducible $q$-ary words is denoted by ${\rm Irr}_{k}(q)$ and the set of $k$-irreducible words of length $n$ is denoted by ${\rm Irr}_{k}(n,q)$.
The $k$-ancestors of ${\bx} \in \Sigma_q^*$ that are $k$-irreducible words are called the $k$-\textit{roots} of 
${\bx}$. The set of $k$-roots of $\bx$ is denoted by $R_{k}({\bx})$.
In a similar fashion, we have the definitions of ${\le}k$-{\em irreducible} words and ${\le}k$-roots of $\bx$, and 
the notation ${\rm Irr}_{\le k}(q)$, ${\rm Irr}_{\le k}(n,q)$, and $R_{\le k}(\bx)$.
\end{definition}


\begin{example}
For alphabet $\Sigma_3$, we have $R_{\le 3}(01012012) = \{012\}$, and $R_{\le 2} (012012) = \{012012\}$.
On the other hand, $R_{\le 4}(012101212) = \{012, 0121012\}$.
\end{example}

As we see, it is possible for a word to have more than one root. 
Jain {\em et al.} determined when a word has exactly one root and 
provided efficient algorithms to compute this unique root.

\begin{proposition}[Jain {\em et al.}\cite{Jain}]\label{prop:jain1}
For any ${\bx} \in \Sigma_q^*$, we have that 
$|R_{k}({\bx})|=1$ for all $k$, and $|R_{\le k}({\bx})|=1$ for all $k\in\{1,2,3\}$. 
Furthermore, there exist linear-time algorithms to compute $R_{k}({\bx})$ for all $k$. 
\end{proposition}

One may easily derive linear-time algorithms to compute $R_{\le k}({\bx})$ for all for $k\in\{1,2,3\}$. 
For completeness, we formally describe these algorithms in Appendix \ref{app:root}.

As it turns out, for certain cases, determining the confusability of two words 
is equivalent to computing the roots for the words.

\newpage

\begin{proposition}[Jain {\em et al.} \cite{Jain}]\label{prop:jain2}
For all ${\bx},{\by} \in \Sigma_q^*$, we have that 
\begin{enumerate}[(i)]
\item $\bx$ and $\by$ are $k$-confusable if and only if $R_{k}({\bx})=R_{k}({\by})$ for all $k$;
\item $\bx$ and $\by$ are ${\le}k$-confusable if and only if $R_{\le k}({\bx})=R_{\le k}({\by})$ for $k\in\{1,2\}$.
\item If $R_{\le 3}({\bx})\ne R_{\le 3}({\by})$, then $\bx$ and $\by$ are not ${\le}3$-confusable.
\end{enumerate}
\end{proposition}

In other words, when $k\in\{1,2\}$, to determine $\bx$ and $\by$ are $\le k$-confusable,
it is both necessary and sufficient to compute the ${\le}k$-roots of $\bx$ and $\by$.
Therefore, applying Proposition \ref{prop:jain1} and the algorithms in Appendix \ref{app:root}, 
we are able to determine whether two words are ${\le}k$-confusable in linear time for $k\in\{1,2\}$.

Unfortunately, when $k=3$, it is no longer sufficient to compute the ${\le}3$-roots of $\bx$ and $\by$.
Specifically, even though $R_{\le 3}({\bx})=R_{\le 3}({\bf y})$,
it is possible that $\bx$ and $\by$ are not ${\le}3$-confusable.
We illustrate this in the next example.

\begin{example}\label{confuse}
Consider ${\bx}=012012$ and ${\by}=011112$ over $\Sigma_3$. 
The words have the same root as $R_{\le 3}({\bx})=R_{\le 3}({\bf y}) = \{012\} $. 
However, $\bx$ and $\by$ are not ${\le}3$-confusable because any ${\le}3$-descendant of ${\bx}$ has a 2 to the left of a 0, 
whereas any ${\le}3$-descendant of ${\by}$ does not.
\end{example}
Therefore, the next smallest open case is where $k=3$. 
A polynomial-time algorithm is implied by the results of Leupold {\em et al.} \cite{Leupold1} and Jain {\em et al.} \cite{Jain2}.
\begin{proposition}[Leupold {\em et al.} \cite{Leupold1}, Jain {\em et al.} \cite{Jain2}]\label{prop:regular}
Let $k\in\{2,3\}$. Let ${\bx} \in \Sigma_q^*$. 
Then $D_{\le k}^{*}({\bx})$ is a regular language. 
Furthermore, if $|{\bx}|=m$, then the deterministic finite automaton that generates $D_{\le k}^{*}({\bx})$ has $O(m)$ vertices and $O(m)$ edges. The number of vertices and edges is independent of the alphabet size.  
\end{proposition}

Therefore, for two words ${\bx}$ and ${\by}$ with $|{\bx}|=m$ and $|{\by}|=n$, 
the language $D_{\le 3}^{*}({\bx}) \cap D_{\le 3}^{*}({\by})$ is regular and 
the corresponding finite automaton has $O(mn)$ vertices and $O(mn)$ edges. 
Hence, to determine if ${\bx}$ and ${\by}$ are confusable, 
it suffices to determine if the language $D_{\le 3}^{*}({\bx}) \cap D_{\le 3}^{*}({\by})$ is nonempty. 
The latter can be done in $O(mn)$ time. 
We improve this running time by providing an algorithm 
that runs in $O(\max\{m,n\})$ time.

\vspace{2mm}

\noindent{\bf Code Construction}. We are interested in determining the maximum possible size of an $(n,{\le}k;q )$-TD code. 
An  $(n,{\le}k;q )$-TD code that achieves this maximum is said to be {\em optimal}.

Motivated by Proposition \ref{prop:jain2}, Jain {\em et al.} used irreducible words to construct tandem-duplication codes. 

\begin{construction}\label{code:jain}
For $k\in\{2,3\}$ and $n \ge k$. An $(n,{\le}k; q)$-TD code ${\cal C}_I(n, {\le}k;q)$ is given by
\begin{equation*}
{\cal C}_I(n,{\le}k;3)=\bigcup_{i=1}^{n} \left\{\xi_{n-i}({\bx})\mid {\bx} \in {\rm Irr}_{\le k}(i,q)\right\}.
\end{equation*} 
Here, $\xi_{i}({\bx})=\bx z^i$, where $z$ is the last symbol of $\bx$. 
In other words, $\xi_{i}({\bx})$ is the sequence ${\bx}$ with its last symbol $z$ repeated $i$ more times.
Furthermore, the size of ${\cal C}_I(n,{\le}k;q)$ is $\sum_{i=1}^n \left|{\rm Irr}_{\le k}(i,q)\right|$.
\end{construction}

It then follows from Proposition \ref{prop:jain2} that ${\cal C}_I(n,{\le}2;q)$ is an optimal $(n,{\le}2;q)$-TD code.
Unfortunately, the TD code ${\cal C}_I(n,{\le}3;q)$ is not an optimal for $n\ge 6$,
and we illustrate this via the following example.

\begin{example}
Consider $n=6$, $q=3$. The code ${\cal C}_I(6,\le 3;3)$ from Construction~\ref{code:jain} is given by
\begin{align*}
{\cal C}_I(6,{\le}3;q)
& = \{aaa aaa\mid a\in\Sigma_3 \} \cup \{abb bbb\mid a, b\in\Sigma_3, a\ne b \}\\
& \cup \{aba aaa,abc ccc\mid a, b,c \in\Sigma_3, a\ne b, b\ne c, a\ne c \}\\
& \cup \{abac cc, abca aa, abcb bb \mid a, b, c\in\Sigma_3, a\ne b, b\ne c, a\ne c \}\\
& \cup \{aba caa, aba cbb, abc abb, abc acc, abc baa \mid a, b,c\in\Sigma_3, a\ne b, b\ne c, a\ne c \}\\
& \cup \{aba cab, aba cba, aba cbc, abc aba, abc acb, abc bab, abc bac \mid a, b,c\in\Sigma_3, a\ne b, b\ne c, a\ne c \}.
\end{align*} 
Hence, $|{\cal C}_I(6,{\le}3;3)|=111$.
However, we may remove from ${\cal C}_I(6,{\le}3;3)$ 
the six codewords  $\{abc ccc \mid a, b,c \in\Sigma_3, a\ne b, b\ne c, a\ne c \}$ and
augment the code with twelve more codewords  $\{abcabc, abbbbc \mid a, b,c \in\Sigma_3, a\ne b, b\ne c, a\ne c \}$.
Then we can check that the new code has size $117$, and 
we later verify using Proposition \ref{prop:irr2-upper} 
that the new code is indeed optimal.
For lengths at least six, we construct TD codes with strictly larger sizes.
\end{example}

\noindent{\bf Upper Bound}. We also study upper bounds on the size of an optimal $(n,{\le}3;q)$-TD code.
By definition, an $(n,{\le}3;q)$-TD code is also an $(n,{\le}2;q)$-TD code. 
Since an optimal $(n,{\le}2;q)$-TD code is provided by Construction~\ref{code:jain}, 
we have the following upper bound on the size of an optimal $(n,{\le}3;q)$-TD code.

\begin{proposition}\label{prop:irr2-upper}
The size of an $(n,{\le}3;q)$-TD code is at most $\sum_{i=1}^n \left|{\rm Irr}_{\le 2}(i,q)\right|$.
\end{proposition}

Proposition \ref{prop:irr2-upper} implies that Construction~\ref{code:jain} is tight for $k=3$ and $n\le 5$.
Using a combinatorial characterization implied by our algorithm,
we improve this upper bound for longer lengths in Section \ref{sec:codes}.

\subsection{Our Contributions}

\begin{itemize}
\item {\bf Confusability}. In Section \ref{sec:confusable}, 
we present sufficient and necessary conditions for two words to be ${\le}3$-confusable and
propose a linear-time algorithm to solve the ${\le}3$-confusability problem. 

\item {\bf Estimates on code sizes}. Using insights gained from Section \ref{sec:confusable}, 
we study the size of tandem-duplication codes in Section \ref{sec:codes}.
We first improve the upper bound given by Proposition \ref{prop:irr2-upper} and 
then construct codes with larger sizes as compared to those given by Construction \ref{code:jain}.
We also provide certain explicit constructions and recursive constructions for tandem-duplication codes.
Furthermore, we determine the sizes of optimal ${\le}3$-tandem-duplication codes
for lengths up to {twenty}.
\end{itemize}


\section{Determining ${\le}3$-Confusability}
\label{sec:confusable}

We derive a linear-time algorithm to determine the ${\le}3$-confusability of two words $\bx$ and $\by$. 
For the sake of simplicity, we omit the use of ``${\le}3$'' and assume confusable, descendant and root to mean ${\le}3$-confusable, ${\le}3$-descendant and ${\le}3$-root, respectively.

Prior to stating the algorithm, we make some technical observations on tandem duplications of length at most three.
The proofs of the lemmas are deferred to Appendix \ref{app:order}.
The first lemma states that given a sequence of tandem duplications, we may reorder the tandem duplications
such that the lengths of the repeats are nonincreasing.

\begin{lemma}\label{lem:reorder}
Let $k_1<k_2\le 3$ and $\bx\in \Sigma_q^*$.
Suppose that $T_{i_2,k_2}\circ T_{i_1,k_1}(\bx)={\bx'}$ for some $i_1,i_2$.
Then there exist integers $3\ge \ell_1\ge \ell_2\ge \cdots\ge \ell_t$ and $j_1,j_2,\ldots, j_t$
such that
\[
     T_{j_t,\ell_t}\circ T_{j_{t-1},\ell_{t-1}}\circ \cdots \circ T_{j_1,\ell_1}({\bx})=\bx'. 
 \]
\end{lemma}

The next lemma states that we may assume the duplications of length two and three 
are performed on segments whose symbols are distinct.
This is because if a duplication is performed on a segment whose symbols are not distinct, 
we may find an equivalent sequence of duplications of strictly shorter lengths.

\begin{lemma}\label{lem:distinct}
Let $k\in \{2,3\}$. Suppose $T_{i,k}(\bx)=\bu\bv\bv\bw$, where $\bx=\bu\bv\bw$, $|\bu|=i$ and $|\bv|=k$.
If the symbols in $\bv$ are not pairwise distinct, then 
there exist integers $k> \ell_1\ge \ell_2$ and $j_1,j_2$
such that
\[
     T_{j_2,\ell_2}\circ T_{j_1,\ell_1}({\bx})=T_{i,k}(\bx). 
 \]
\end{lemma}

Henceforth, we use the notation ${\bx} \xRightarrow[k_d]{*} {\by}$
to denote that there is a sequence of tandem duplications of length $k$ of {\em distinct} symbols from ${\bx}$ to ${\by}$. 
It follows from definition that ${\bx} \xRightarrow[1]{*} {\by}$
is equivalent to ${\bx} \xRightarrow[1_d]{*} {\by}$.
The next lemma is immediate from Lemmas \ref{lem:reorder} and \ref{lem:distinct}.

\begin{lemma}\label{lem:ordering}
Suppose that ${\bx} \xRightarrow[\le3]{*} {\bx'}$.
Then there exist ${\bx_1}$ and ${\bx_2}$ such that
${\bx} \xRightarrow[3_d]{*} {\bx_2} \xRightarrow[2_d]{*} {\bx_1} \xRightarrow[1_d]{*} {\bx'}.$
Furthermore, $\bx_1$ and $\bx_2$ are uniquely determined by
 $\bx_1=R_{1}(\bx')$ and $\bx_2=R_{\le 2}(\bx')$.
\end{lemma}

Equipped with this lemma, we state the following theorem that provides a 
necessary and sufficiency condition for two words to be confusable.

\begin{theorem}
\label{confusable}
Two words ${\bx}$ and ${\by}$ are ${\le}3$-confusable if and only if there exist ${\bx'}$ and ${\by'}$ such that 
\[{\bx} \xRightarrow[3_d]{*} {\bx'},~{\by} \xRightarrow[3_d]{*} {\by'}
\mbox{ and }R_{\le 2}({\bx'})=R_{\le 2}({\by'}).\]
\end{theorem}

\begin{proof}
		Suppose that ${\bx}$ and ${\by}$ are confusable. 
		Then there exists ${\bz}$ such that 
		${\bx} \xRightarrow[\le 3]{*} {\bz}$ and ${\by} \xRightarrow[\le 3]{*} {\bz}$.
		By Lemma~\ref{lem:ordering}, there exist ${\bx_1}, {\bx_2}, {\by_1}, {\by_2}$ such that
		${\bx} \xRightarrow[3_d]{*} {\bx_2} \xRightarrow[2_d]{*} {\bx_1} \xRightarrow[1_d]{*} {\bz}$ and 
		${\by} \xRightarrow[3_d]{*} {\by_2} \xRightarrow[2_d]{*} {\by_1} \xRightarrow[1_d]{*} {\bz}$. 
		Hence we set ${\bx'}={\bx_2}$ and ${\by'}={\by_2}$. Note that by Proposition~\ref{prop:jain2}, $R_{\le 2}({\bx'})=R_{\le 2}({\by'})$ since 
		$\bx'$ and $\by'$ are ${\le}2$-confusable.
		
		Conversely, suppose that there exist ${\bx'}$ and ${\by'}$ such that
		${\bx} \xRightarrow[3_d]{*} {\bx'}$, ${\by} \xRightarrow[3_d]{*} {\by'}$, and $R_{\le 2}({\bx'})=R_{\le 2}({\by'})$. 
		Since ${\bx'}$ and ${\by'}$ have a common ${\le}2$-root, it follows from Proposition~\ref{prop:jain2} that ${\bx'}$ and ${\by'}$ are ${\le}2$-confusable. Thus, ${\bx}$ and ${\by}$ are confusable.
	\end{proof}

In conclusion, for two words ${\bx}$ and ${\by}$ to be confusable, 
it is equivalent to checking if there are tandem duplications of length three that make 
their descendants ${\bx'}$ and ${\by'}$ have the same ${\le}2$-root.
We make use of this important fact in the next section.

\subsection{Strategy behind Algorithm}

We outline our strategy to determine the confusability of two words $\bx$ and $\by$. Recall from Proposition \ref{prop:jain2}(iii), if  ${\bx}$ and ${\by}$ have different roots, 
then they are not confusable.
Hence, our first step is to determine $R_{\le 3}(\bx)$ and $R_{\le 3}(\by)$ and see if the roots are equal.
If the roots are not equal, we immediately conclude that $\bx$ and $\by$ are not confusable.
Therefore, the nontrivial task is to determine confusability when $R_{\le 3}(\bx)=R_{\le 3}(\by)$.
Henceforth, we assume that the {\em common root of $\bx$ and $\by$ is $\br$}.

Suppose that $\br$ contains at most two distinct symbols.
Then from Lemma \ref{lem:ordering}, 
we have that $\br\xRightarrow[\le 2]{*}\bx$ and $\br\xRightarrow[\le 2]{*}\by$.
In other words, $\br$ is also the ${\le}2$-root of both $\bx$ and $\by$.
Applying Proposition \ref{prop:jain2}, we have that $\bx$ and $\by$ are ${\le}2$-confusable and so, ${\le}3$-confusable.
Therefore, it remains to consider the case where $\br$ contains {\em at least three distinct} symbols.

Given $\br$, $\bx$ and $\by$, our next step is to compute certain proper prefixes ${\rm *Pref}(\br,\bx)$ and ${\rm *Pref}(\br,\by)$ of $\bx$ and $\by$, respectively, such that {\em under certain conditions} the following holds:
\[ \left(\bx \mbox{ and } \by \mbox{ are confusable}\right)
\mbox{ if and only if }
\left(\bx\setminus{\rm *Pref}(\br,\bx)\mbox{ and } \by\setminus{\rm *Pref}(\br,\by) \mbox{ are confusable}\right)\,.\]

Here, $\bx\setminus \bz$ denotes the word we obtain by removing the prefix $\bz$ from $\bx$.
Since the words on the righthand side are strictly shorter than the words on the lefthand side, we may repeat this process for the shorter words until the common root of the words has less than three distinct symbols.

To define the prefixes ${\rm *Pref}(\br,\bx)$ and ${\rm *Pref}(\br,\by)$, we introduce more notation.

Now, $\br$ contains at least three distinct symbols.
Since $\br\in{\rm Irr}_{\le 3}(q)$,
we easily check that either $r_1r_2r_3$ or $r_2r_3r_4$ is a substring with three distinct symbols.
Define ${\rm main}(\br)$ to be the first substring in $\br$ with three distinct symbols and 
we define the {\em region of $\br$}, denoted by ${\rm Reg}(\br)$, to be a certain prefix of $\br$ according to the following rule.

If $r_1=r_3$, we set
\[   
{\rm main}({\br}) \triangleq r_2 r_3 r_4, \mbox{ and }
{\rm Reg}({\br}) \triangleq 
     \begin{cases}
       r_1 r_2 r_3 r_4=r_1r_2r_1r_4, &\text{if } r_2 \neq r_5, \\
       r_1 r_2 r_3 r_4 r_5=r_1r_2r_1r_4r_2, &\text{if } r_2 = r_5, r_3 \neq r_6, \\
       r_1 r_2 r_3 r_4 r_5 r_6=r_1r_2r_1r_4r_2r_1, &\text{otherwise.}
     \end{cases}
\]

If $r_1 \neq r_3$, we set
\[   
{\rm main}({\br}) \triangleq r_1 r_2 r_3, \mbox{ and }
{\rm Reg}({\br}) \triangleq 
     \begin{cases}
       r_1 r_2 r_3, &\text{if } r_1 \neq r_4, \\
       r_1 r_2 r_3 r_4=r_1 r_2 r_3r_1, &\text{if } r_1 = r_4, r_2 \neq r_5, \\
       r_1 r_2 r_3 r_4 r_5=r_1 r_2 r_3r_1r_2, &\text{otherwise.}
     \end{cases}
\]

Intuitively, the region ${\rm Reg}(\br)$ of $\br$ is defined as above so that the following properties of the region  hold.
\begin{lemma}\label{lem:reg}
Suppose the root $\br$ has three distinct symbols. Define ${\rm Reg}(\br)$ as above.
Then ${\rm Reg}(\br)$ can be written as $\bw(abc)^\ell ab$, where $a,b,c$ are distinct symbols, $\ell\in\{0,1\}$ and 
$\bw$ is prefix of length at most three over the alphabet $\{a,b,c\}$.
Furthermore, if ${\rm Reg}(\br)\xRightarrow[3_d]{*} {\bz}$, then $\bz$ can be written as $\bw(abc)^m ab$
with $m\ge \ell$. Also, the symbol in $\br$ after the prefix ${\rm Reg}(\br)$ is not $c$.
\end{lemma}

\begin{example}
Suppose that ${\br}=010201$. Then ${\rm main}({\br})=102$ and ${\rm Reg}({\br})=0102$.
Indeed, we can write ${\rm Reg}({\br})$ as $\bw(abc)^\ell ab$ with $\bw=01$, $\ell=0$, and $abc=021$.
We also check that the next symbol after ${\rm Reg}({\br})$ in ${\br}$ is $0$, which is not $c=1$.
\end{example}

Given ${\rm Reg}(\br)=\bw(abc)^\ell ab$, we consider the longest prefix of ${\bx}$ that is generated from ${\rm Reg}(\br)$ through a sequence of tandem duplications of length at most three.
We call this prefix the {\em extended region of $\br$ in $\bx$} and 
denote it with ${\rm Ext}({\rm Reg}(\br),{\bx})$. 
Finally, we define our special prefix ${\rm *Pref}(\br,\bx)$. 
Suppose that ${\rm Ext}({\rm Reg}(\br),{\bx})=p_1p_2\cdots p_t=\bp$
and $p_i$ is the last appearance of $a$ in $\bp$.
The {\em special prefix} ${\rm *Pref}(\br,\bx)$ is given by $p_1p_2\cdots p_{i-1}$.

\begin{example} Let ${\br}=010201$ and ${\bx}=01102021020120111$. 
Recall that ${\rm Reg}({\br})=0102$ and $a=0$. Then 
\[{\rm Ext}(0102,{\bx})=0110202102 \mbox{ and } {\rm*Pref}(\br, \bx)=01102021.\]
\end{example}

We summarize the relationships of $\br$, $\bx$, ${\rm Reg}({\br})$, ${\rm Ext}({\rm Reg}(\br),{\bx})$, and ${\rm *Pref}(\br,\bx)$ in Figure \ref{fig:prefixes} and 
we are ready to state the main characterization theorem. 

\begin{figure}[!t]
\begin{center}
\vspace{-2mm}
\includegraphics[width=15cm]{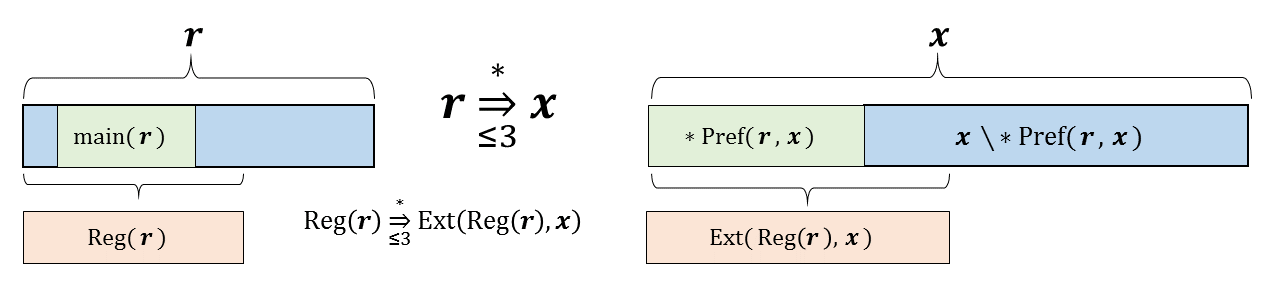}
\end{center}
\caption{Prefixes Used to Determine Confusability}
\label{fig:prefixes}
\end{figure}

\begin{theorem}\label{thm:main}
		Suppose that ${\bx}$ and ${\by}$ are two words such that ${\br}=R_{\le 3}({\bx})=R_{\le3}({\by})$ and
		$\br$ contains at least three distinct symbols.
		Set 
		\begin{align*}
			\bp&={\rm Ext}({\rm Reg}({\br}),{\bx}),&
			\bq&={\rm Ext}({\rm Reg}({\br}),{\by})\\
			\ba&={\bx}\setminus{\rm *Pref}({\br},{\bx})& 
			\bb&={\by}\setminus{\rm *Pref}({\br},{\by}).
		\end{align*}
		Then ${\bx}$ and ${\by}$ are confusable if and only if ${\bp}$ and ${\bq}$ are confusable, and ${\ba}$ and ${\bb}$ are confusable.
	\end{theorem}

	\begin{proof}
		

		From Lemma \ref{lem:reg}, we may set ${\rm Reg}(\br)=\bw(abc)^\ell ab$.
		Since ${\rm Reg}(\br)\xRightarrow[\le 3]{*}\bp$, Lemma \ref{lem:ordering}
		implies that ${\rm Reg}(\br)\xRightarrow[3_d]{*}R_{\le 2}(\bp)$. So,
		$R_{\le 2}(\bp)=\bw(abc)^mab$ for some $m\ge \ell$, and $R_{\le 2}(\bq)=\bw(abc)^jab$ for some $j\ge \ell$.
		
		Suppose that ${\bx}$ and ${\by}$ are confusable. By Theorem~\ref{confusable}, there exist $3_d$-descendants of ${\bx}$ and ${\by}$ namely ${\bf x'}$ and ${\bf y'}$, such that
		$R_{\le 2}({\bx'})=R_{\le 2}({\by'})$.
		Set 
		\begin{align*}
			\bp'&={\rm Ext}({\rm Reg}({\br}),{\bx'}),&
			\bq'&={\rm Ext}({\rm Reg}({\br}),{\by'})\\
			\ba'&={\bx'}\setminus{\rm *Pref}({\br},{\bx'})& 
			\bb'&={\by'}\setminus{\rm *Pref}({\br},{\by'}).
		\end{align*}
		
		Since the overlapping region between $\bp$ and $\ba$ must be of the form 
		$ab\cdots b$, each duplication of distinct triplets from ${\bx}$ to ${\bx'}$ must be either entirely in $\bp$ or in $\ba$. In other words,
		$\bp \xRightarrow[3_d]{*} \bp'$ and $\ba \xRightarrow[3_d]{*} \ba'$, and likewise $\bq \xRightarrow[3_d]{*} \bq'$ and $\bb \xRightarrow[3_d]{*} \bb'$.
		Looking at the overlapping region between $\bp'$ and $\ba'$, 
		we can observe that $R_{\le 2}({\bx'}) = R_{\le 2}(\bp' \ba')$. 
		Then ${\rm Reg}(\br)\xRightarrow[\le 3]{*}\bp'$ and Lemma \ref{lem:ordering}
		imply that ${\rm Reg}(\br)\xRightarrow[3_d]{*}R_{\le 2}(\bp)$ and so,
		$R_{\le 2}(\bp')=\bw(abc)^iab$ for some $i\ge m$.
		
		Next, $\bw(abc)^i R_{\le 2}(\ba')$ is a ${\le}2$-irreducible word, 
		because $R_{\le 2}(\ba')$ starts with $ab$ 
		and by definition, $\bw(abc)^i$ ends with $c$ but not $cac$.
		Therefore, $\bw(abc)^i R_{\le 2}(\ba')$ is the ${\le}2$-root of $\bx'$, or,
		$R_{\le 2}({\bx'})=\bw(abc)^i R_{\le 2}(\ba')$ and likewise $R_{\le 2}({\by'})=\bw(abc)^h R_{\le 2}(\bb')$ for some $h\ge j.$
		
		Since the first symbol of $R_{\le 2}({\bx'}\setminus \bp')$ and 
		$R_{\le 2}({\by'}\setminus \bq')$ cannot be $c$,
		neither $R_{\le 2}(\ba')$ nor $R_{\le 2}(\bb')$ start with $abc$. 
		Therefore, since $R_{\le 2}({\bx'})=R_{\le 2}({\by'})$,
		we can deduce that $h=i$ and $R_{\le 2}(\ba')=R_{\le 2}(\bb')$. Since $h=i$, we have $R_{\le 2}(\bp')=R_{\le 2}(\bq')$. Together with $\bp \xRightarrow[3_d]{*} \bp'$, $\bq \xRightarrow[3_d]{*} \bq'$, by Theorem \ref{confusable}, we can have $\bp$ and $\bq$ are confusable.
		Since $\ba \xRightarrow[3_d]{*} \ba'$, $\bb \xRightarrow[3_d]{*} \bb'$, and  $R_{\le 2}(\ba')=R_{\le 2}(\bb')$, we can also conclude by Theorem \ref{confusable} that $\ba$ and $\bb$ are confusable.
		
		Suppose that ${\bp}$ and ${\bq}$ are confusable, and $\ba$ and $\bb$ are confusable. Hence ${\bp}\ba$ and ${\bq}\bb$ are confusable. Note that ${\bp}\ba$ is a descendant of ${\bx}$ and ${\bq}\bb$ is a descendant of ${\by}$, hence ${\bx}$ and ${\by}$ are also confusable.
	\end{proof}
	
	We can efficiently determine whether the prefixes $\bp$ and $\bq$ are confusable. 
	When the substring ${\rm main}(\br)$ appears the same number of times in $R_{\le 2}(\bp)$ as in $R_{\le 2}(\bq)$, 
	we may assume that $R_{\le 2}(\bp)=\bw(abc)^mab=R_{\le 2}(\bq)$.
	Hence, ${\bp}$ and ${\bq}$ are confusable by Proposition \ref{prop:jain2}.
	 In particular, we have Corollary \ref{cor:samecount}.  
	
	Consider $\bz=abc$, where $a$, $b$, $c$ are three distinct symbols.
Let ${\rm Count}({\bz},{\bx})$ denote the number of times ${\bz}$ appears as a substring in ${\bx}$,
and ${\rm Count}({\rm rot}(\bz),{\bx})$ denote the number of times $abc$, $bca$, or $cab$ appears as a substring in ${\bx}$.

\begin{algorithm}[h!]\footnotesize
\caption{{\sc Confuse}($\bx,\by$)}
\label{alg:confuse}
\begin{algorithmic}[1]
\small
\Require $\bx, \by \in \Sigma_q^*$
\Ensure {\tt true} if ${\bx}$ and ${\by}$ are confusable, {\tt false} otherwise. 
\State $\br\gets R_{\le 3}({\bx})$
\State $\br'\gets R_{\le 3}({\by})$

 \If{$\bf r \neq \bf r'$} 
 \State \Return{{\tt false}} 
\Else 
\State \Return{{\sc Confuse-Recursion}($\bx,\by,\br$)}
 \EndIf
\Statex
\Procedure {Confuse-Recursion}{$\bx$,$\by$,$\br$}
 
 \If{$\br$ has at most two distinct symbols}
\State \Return {\tt true}
\Else
\State{Compute ${\rm main}(\br)$ and ${\rm Reg}(\br)$}
\State{$\bp\gets {\rm Ext}({\rm Reg}(\br),\bx)$, \hspace{20mm}$\bq\gets {\rm Ext}({\rm Reg}(\br),\by)$}
\State{$\ba\gets \bx\setminus{\rm *Pref}(\br,\bx)$, \hspace{19mm} $\bb\gets \by\setminus{\rm *Pref}(\br,\by)$}
\State{$C_\bp\gets {\rm Count}({\rm main}({\br}),R_{\le_2}(,\bp))$,  \hspace{2mm}
$C_\bq\gets {\rm Count}({\rm main}({\br}),R_{\le_2}(\bq))$}
\State{$\br^*\gets \br\setminus \bw(abc)^\ell$, where ${\rm Reg}(\br)=\bw(abc)^\ell ab$}
\If{$C_\bp=C_\bq$}
\State 
\Return {{\sc Confuse-Recursion}($\ba,\bb,\br^*$)}
\ElsIf {$C_\bp<C_\bq$ and ${\rm Count}({\rm rot}({\rm main}({\br})),\bp)>0$}
\State 
\Return {{\sc Confuse-Recursion}($\ba,\bb,\br^*$)}
\ElsIf {$C_\bp>C_\bq$ and ${\rm Count}({\rm rot}({\rm main}({\br})),\bq)>0$}
\State
\Return {{\sc Confuse-Recursion}($\ba,\bb,\br^*$)}
\Else
\State \Return {\tt false}
\EndIf
 \EndIf
 \EndProcedure
\end{algorithmic}
\end{algorithm}

	\begin{corollary}\label{cor:samecount}
		Suppose that ${\bx}$ and ${\by}$ are two words such that ${\br}=R_{\le 3}({\bx})=R_{\le3}({\by})$ and
		$\br$ contains at least three distinct symbols.
		Set 
		$\bp={\rm Ext}({\rm Reg}({\br}),{\bx})$ and $\bq={\rm Ext}({\rm Reg}({\br}),{\by})$.
		Suppose further that 
		${\rm Count}({\rm main}({\br}),R_{\le2}(\bp))={\rm Count}({\rm main}({\br}),R_{\le 2}(\bq))$.
		Then ${\bx}$ and ${\by}$ are confusable if and only if 
		${\bx}\setminus{\rm *Pref}({\br},{\bx})$ and ${\by}\setminus{\rm *Pref}({\br},{\by})$ are confusable.
	\end{corollary}
	
	On the other hand, when the number of appearances is not equal, we may modify the proof
	to obtain the following corollary. 
	A detailed proof of the corollary is deferred to Appendix \ref{app:cor1}.
	
	\begin{corollary}\label{cor:main}
		Suppose that ${\bx}$ and ${\by}$ are two words such that ${\br}=R_{\le 3}({\bx})=R_{\le3}({\by})$ and
		$\br$ contains at least three distinct symbols.
		Set 
		$\bp={\rm Ext}({\rm Reg}({\br}),{\bx})$ and $\bq={\rm Ext}({\rm Reg}({\br}),{\by})$.
		Suppose further that 
		${\rm Count}({\rm main}({\br}),R_{\le2}(\bp))<{\rm Count}({\rm main}({\br}),R_{\le 2}(\bq))$.
		Then ${\bx}$ and ${\by}$ are confusable if and only if 
		${\rm Count}({\rm rot}({\rm main}({\br})),\bp)>0$ and ${\bx}\setminus{\rm *Pref}({\br},{\bx})$ and ${\by}\setminus{\rm *Pref}({\br},{\by})$ are confusable.
	\end{corollary}

\subsection{Confusability Algorithm}\label{sec:implementation}

Based on the ideas in the previous subsection,
we formally describe our algorithm {\sc Confuse} in Algorithm \ref{alg:confuse}  
and demonstrate that the running time is linear.
In the following analysis, we assume the definitions in Algorithm \ref{alg:confuse}.

\begin{itemize}
\setlength\itemsep{2mm}
\item \underline{${\rm main}(\br)$ and ${\rm Reg}(\br)$ can be computed in constant time}. This is clear from definition.

\item \underline{The prefixes $\bp$ and $\bq$ can be computed in time linear in $|\bp|$ and $|\bq|$, respectively}.
From Proposition \ref{prop:regular}, we have that $D^*_{\le 3}({\rm Reg}(\br))$
is a regular language. 
Since the length of ${\rm Reg}(\br)$ is at most six, the finite automaton generating the language has a constant number of vertices. 
Hence, the longest prefixes of $\bx$ and $\by$ that are accepted by the finite automaton are precisely $\bp$ and $\bq$, respectively.
In Appendix \ref{app:prefixes}, we show that the sum of the lengths of all prefixes that are computed in all recursive calls is at most $3(|\bx|+|\by|)$. 
Therefore, the computation of all prefixes takes linear time.

\item \underline{In the recursive calls, the root $\br^*$ can be computed in constant time}. As mentioned earlier, the first step is to compute the roots of $\bx$ and $\by$ in linear time. 
When the roots are equal to $\br$ and certain conditions are met, 
we make the recursive call {\sc Confuse-Recursion}. 
If the roots are recomputed for the shorter words $\ba$ and $\bb$, then the overall running time for {\sc Confuse} is quadratic.
To avoid this, we observe that the roots for the shorter words are always equal and
in fact, the common root $\br^*$ of $\ba$ and $\bb$ may be easily inferred from $\br$.
More precisely, if ${\rm Reg}(\br)=\bw(abc)^\ell ab$, then $\br^*= \br\setminus \bw(abc)^\ell$.

\item \underline{In the recursive calls, determining $\br$ has at most two symbols can be done in constant time}.
If $\br$ has length at least four, then $\br$ necessarily has at least three distinct symbols.
Otherwise, it contradicts the fact that $\br$ is irreducible. 
Therefore, it suffices to check if the first four symbols $\br$ contain three distinct symbols.
\end{itemize}

In summary, from the above observations, 
we conclude that the running time is $O(\max\{|\bx|,|\by|\})$.

\section{${\le}3$-Tandem-Duplication Codes}\label{sec:codes}

We use insights gained from the previous section to construct  $(n,{\le}3;q)$-TD codes. Motivated by the concept of roots, we consider a ${\le}3$-irreducible word $\br$
and we say that a $(n,{\le}3;q)$-TD code $\cal C$ is an $(n,{\le}3;\br)$-TD code if 
all words in $\cal C$ belong to $D_{\le 3}^*(\br)$.

For $\br\in{\rm Irr}_{\le3}(i,q)$ with $i\le n$, 
suppose that ${\cal C}(n,\br)$ is an $(n,{\le}3;\br)$-TD code.
Then Proposition \ref{prop:jain2} implies that 
$\bigcup_{\br\in{\rm Irr}_{\le3}(i,q),\, i\le n}{\cal C}(n,\br)$ is an $(n,{\le}3;q)$-TD code.
Trivially, $\{\xi_{n-i}({\br})\}$  is an $(n,{\le}3;\br)$-TD code for all ${\bx} \in {\rm Irr}_{\le 3}(i,q)$.
Taking the union of these codes, we recover Construction \ref{code:jain}.

Therefore, in the rest of this section, our objective is 
to provide estimates on the size of an optimal $(n,{\le}3;\br)$-TD code for fixed ${\le}3$-irreducible word $\br$.
To simplify our discussion, we focus on the case where $q=3$ and let 
$T(n)$ and $T(n,\br)$ to denote the sizes of an optimal $(n,{\le}3;3)$-TD code
and an optimal $(n,{\le}3;\br)$-TD code, respectively.

Pick $\bx\in D_{{\le}3}^*(\br)$. Using concepts from Section \ref{sec:confusable},
we define the following combinatorial characterisation of $\bx$. 
Set $\bx_1=\bx$ and $\br_1=\br$.
For $i\ge 2$, set
$\bx_i=\bx_{i-1}\setminus *{\rm Pref}(\br_{i-1},\bx_{i-1}) \mbox{ and } \br_i= R_{{\le}3}(\bx_i).$
We terminate this recursion when $\br_{m+1}$ has less than three distinct symbols.
If $\br_{m}$ is the last root to have three distinct symbols, we say that $\br$ has $m$ {\em regions} and
for $1\le i\le m$, define
\[ \bp_i={\rm Ext}({\rm Reg}(\br_i),\bx_i),\, \bt_i={\rm main}(\br_i),\, c_i={\rm Count}(\bt_i,R_{{\le}2}(\bp_i)), \mbox{ and }
\delta_i=
\begin{cases}
+, &\mbox{if }{\rm Count}({\rm rot}(\bt_i),\bp_i)>0,\\
-, &\mbox{otherwise.}
\end{cases}\]

For the word $\bx$, we define its {\em label} by
${\rm Label}(\bx)=(\br,(c_1,\delta_1),(c_2,\delta_2),\ldots,(c_m,\delta_m))$.

\begin{example}
Consider $\br=01210$. Then $\br$ has two regions with $\bt_1=012$ and $\bt_2=210$.
Consider also the words $01210210$, $01201210$, and $01111210$ that belongs to $D_{{\le}3}^*(01210)$.
Then their labels are as follow.
\begin{align*}
{\rm Label}(01210210)&=(01210, (1,+),(2,+)),\\
{\rm Label}(01201210)&=(01210, (2,+),(1,+)),\\
{\rm Label}(01112110)&=(01210, (1,-),(1,-)).
\end{align*}
 Then it follows from Theorem \ref{thm:main}, Corollaries \ref{cor:samecount} and \ref{cor:main} that 
 $01210210$ and $01210210$ are confusable, while
 $01210210$ and $01112110$ are not confusable.
\end{example}

More generally, the following proposition is immediate from 
Theorem \ref{thm:main}, Corollaries \ref{cor:samecount} and \ref{cor:main}.

\begin{proposition}\label{prop:label}
Let $\bx$ and $\by$ belong to $D_{{\le}3}^*(\br)$.
Suppose that 
\begin{align*}
{\rm Label}(\bx)&\triangleq\left(\br, \left(c_1^{(\bx)},\delta_1^{(\bx)}\right), \left(c_2^{(\bx)},\delta_2^{(\bx)}\right), \ldots, \left(c_m^{(\bx)},\delta_m^{(\bx)}\right)\right),\\
{\rm Label}(\by)&\triangleq\left(\br, \left(c_1^{(\by)},\delta_1^{(\by)}\right), \left(c_2^{(\by)},\delta_2^{(\by)}\right), \ldots, \left(c_m^{(\by)},\delta_m^{(\by)}\right)\right).
\end{align*}
Then $\bx$ and $\by$ are not ${\le}3$-confusable if and only if
for some $1\le i\le m$, 
\[ \left(c_i^{(\bx)}<c_i^{(\by)} \mbox{ and } \delta_i^{(\bx)}=-\right)
\mbox{ or }
\left(c_i^{(\bx)}>c_i^{(\by)} \mbox{ and } \delta_i^{(\by)}=-\right).\]
\end{proposition}

Immediate from Proposition \ref{prop:label} are certain sufficient conditions for two words to be confusable.

\begin{corollary}\label{cor:label}
Let $\bx, \by\in D_{{\le}3}^*(\br)$ and their labels ${\rm Label}(\bx)$ and ${\rm Label}(\by)$ 
be as defined in Proposition \ref{prop:label}.
\begin{enumerate}[(i)]
\item If $c_i^{(\bx)}=c_i^{(\by)}$ for all $i$, then $\bx$ and $\by$ are ${\le}3$-confusable.
\item If $\delta_i^{(\bx)}=\delta_i^{(\by)}=+$ for all $i$, then $\bx$ and $\by$ are ${\le}3$-confusable.
\end{enumerate}
\end{corollary}

From Corollary \ref{cor:label}, we use the number of integer solutions to certain equations as 
an upper bound for $T(n,\br)$. As this combinatorial argument is fairly technical, 
we state the upper bound and defer the proof to Appendix \ref{app:upper}.

\begin{proposition}\label{prop:upper}
Let $i\le n$. Suppose that $\br\in {\rm Irr}_{{\le}3}(i,3)$ has $m$ regions. 
Then 
\[ T(n,\br)\le U(n,i,m)\triangleq
\begin{cases}
\binom{(n-i)/3+m}{m}-\binom{(n-i)/3+m-1}{m-1}+1, & \mbox{if 3 divides $n-i$},\\
\binom{\floor{(n-i)/3}+m}{m}, &\mbox{otherwise.}
\end{cases}
\]
\end{proposition}


Immediate from Proposition \ref{prop:upper} is that $T(|\br|,\br)=T(|\br|+1,\br)=T(|\br|+2,\br)=1$.
Next, we provide lower bounds for $T(n,\br)$ by constructing $(n,{\le}3;\br)$-TD codes.
We first show that there exists an $(n,{\le}3;\br)$-TD code of size two. 

\begin{construction}\label{code:2}
Let $\br\in{\rm Irr}_{\le 3}(i,3)$ for some $i\ge 4$. Set
\[{\cal C}_{2}\triangleq 
\begin{cases}
\{r_1r_2r_3r_1r_2r_3r_4\cdots r_i,\hspace{4mm} r_1r_2r_2r_3r_3r_4r_4\cdots r_i\},
&\mbox{if $r_1\ne r_3$,}\\
\{r_1r_2r_1r_4r_2r_1r_4r_5\cdots r_i,\hspace{1mm} r_1r_2r_1r_1r_4r_4r_5r_5\cdots r_i\},
&\mbox{if $r_1= r_3$ and $i\ge 5$}\\
\{r_1r_2r_1r_4r_2r_1r_4,\hspace{12mm} r_1r_2r_1r_1r_4r_4r_4\},
&\mbox{if $r_1= r_3$ and $i=4$}\\
\end{cases}
\]
Then ${\cal C}_{2}$ is an $(i+3,{\le}3;\br)$-TD code of size two. 
By Proposition \ref{prop:upper}, $\C_2$ is optimal.
\end{construction}

\begin{proof}Let ${\cal C}_2=\{\bx, \by\}$ and the labels of the words be as given in Proposition \ref{prop:label}. 
Observe that in all cases, 
$\left(c_1^{(\bx)},\delta_1^{(x)}\right)=(2,+)$ and $\left(c_1^{(\by)},\delta_1^{(y)}\right)=(1,-)$.
Proposition \ref{prop:label} then implies that $\bx$ and $\by$ are not confusable.
\end{proof}

Following the above construction, we have a improvement to Construction \ref{code:jain}.

\begin{corollary}For $n\ge 6$, we have that
\[T(n)\ge \sum_{\substack{1\le i\le 3, {\rm~or}\\ n-2\le i\le n}} {\rm Irr}_{\le 3}(i,3)
+2\sum_{i=4}^{n-3}{\rm Irr}_{\le 3}(i,3).\]
\end{corollary}

Proposition \ref{prop:label} suggests that we examine code constructions according to the number of regions of $\br$.
When $|\br|\le 2$, the number of regions is zero and hence $T(n,\br)=1$.
We consider the smallest nontrivial case where $\br$ has exactly one region.
Without loss of generality, assume that $\br\in R\triangleq\{012,0120,01201,1012,10120,101201$, $0121,01202,012010,10121,101202,1012010\}$, and so, $\bt_1=012$. 
In this case, we have the following construction whose proof is deferred to Appendix \ref{app:codes}.

\begin{construction}\label{code:oneregion}
For $\br\in R$ and $\ell\ge 1$, define the words $\bx(\br,\ell)$ and $\bz(\br,\ell)$ with the following rule.
\vspace{1mm}

\begin{center}
\begin{tabular}{|l|l|l||l|l|l|}
\hline
$\br$ & $\bx(\br,\ell)$ & $\bz(\br,\ell)$ &
$\br$ & $\bx(\br,\ell)$ & $\bz(\br,\ell)$ \\ \hline
$012$ & $0(112200)^{\ell-1} 112$ & $(012)^\ell$ &
$1012$ & $10(112200)^{\ell-1} 112$ & $1(012)^\ell$ \\ \hline
$0120$ & $0(112200)^{\ell-1} 11220$ & $(012)^\ell0$ &
$10120$ & $10(112200)^{\ell-1} 11220$ & $1(012)^\ell0$ \\ \hline
$01201$ & $0(112200)^{\ell-1} 1122001$ & $(012)^\ell01$ &
$101201$ & $10(112200)^{\ell-1} 1122001$ & $1(012)^\ell01$ \\ \hline
$0121$ & $0(112200)^{\ell-1} 1121$ & $(012)^\ell1$ &
$10121$ & $10(112200)^{\ell-1} 1121$ & $1(012)^\ell1$ \\ \hline
$01202$ & $0(112200)^{\ell-1} 112202$ & $(012)^\ell02$ &
$101202$ & $10(112200)^{\ell-1} 112202$ & $1(012)^\ell02$ \\ \hline
$012010$ & $0(112200)^{\ell-1} 11220010$ & $(012)^\ell010$ &
$1012010$ & $10(112200)^{\ell-1} 11220010$ & $1(012)^\ell010$ \\ \hline
\end{tabular}
\end{center}
%
For $\br\in R$ and $n\ge |\bx(\br,2)|$, set $\ell_z=\floor{(n-|\br|)/3}+1$ and
\[{\cal C}_{\br}(n)
\triangleq \left\{\xi_{n-|\bx(\br,\ell)|}(\bx(\br,\ell)):|\bx(\br,\ell)| \le n\right\}
\cup \left\{\xi_{n-|\bz(\br,\ell_z)|} (\bz(\br,\ell_z)) \right\}.
\]
Then ${\cal C}_{\br}(n)$ is an $(n,{\le}3;\br)$-TD code. 
Furthermore, ${\cal C}_{\br}(n)$ in optimal. 
Therefore, if we set $n_2=|\bx(r,2)|$,
we have
\[ T(\br,n) = 
\begin{cases} 
\floor{\frac{n-n_2}{6}}+3 &\mbox{if $n\ge n_2$},\\
2, &\mbox{if $|\br|+3\le n< n_2$},\\
1, &\mbox{otherwise.}
\end{cases}
\]
\end{construction}

Hence, the value of $T(n,\br)$ is completely determined whenever $\br$ has at most one region.
Combining this result with Proposition \ref{prop:upper}, we have the following theorem.

\begin{theorem}\label{thm:upper}
Let $I(i,m)$ denote the number of irreducible words in ${\rm Irr}_{{\le}3}(i,3)$ with exactly $m$ regions, and 
$U(n,i,m)$ be as defined in Proposition \ref{prop:upper}.
Then 
\begin{equation}
T(n)\le \sum_{\br \in R} T(n,\br) +\sum_{i=5}^n \sum_{m=2}^{i} I(i,m)U(n,i,m).\label{eq:upper}
\end{equation}
\end{theorem}



In addition to the above constructions, 
we construct tandem-duplication codes for small lengths by searching for them exhaustively.
Specifically, fix $\br\in{\rm Irr}_{\le 3}(i,3)$ and $n\ge i$, we construct the graph ${\cal G}(n,\br)$,
whose vertices correspond to the set of all labels of descendants of $\br$ of length $n$, or 
$\{\bL: \bL={\rm Label}(\bx) \mbox{ and } \bx\in \Sigma_3^n\cap D^*_{\le 3}(\br)\}$.
Two vertices or labels $\bL_1$ and $\bL_2$ are connected in ${\cal G}(n,\br)$ if and only if 
the words whose labels are $\bL_1$ and $\bL_2$ are not ${\le}3$-confusable. 
Hence, a clique of size $M$ in the graph ${\cal G}(n,\br)$ 
correspond to an $(n,{\le}3;\br)$-TD code of size $M$ and 
$T(n,\br)$ is the maximum size of a clique in ${\cal G}(n,\br)$.


We use the exact algorithm {\tt MaxCliqueDyn} \cite{Konc} to 
determine the maximum size of the clique in these graph ${\cal G}(n,\br)$ for ${n\le 20}$.
Since $T(n)=\sum_{\br\in{\rm Irr}_{\le3}(i,3),\, i\le n}T(n,\br)$, 
we tabulate the results $T(n)$ in Table \ref{table:tn}.

\begin{remark}\hfill
\begin{itemize}
\item Such a method to compute $T(n,\br)$ is only possible because 
we have developed the algorithm to determine confusability in Section \ref{sec:confusable}.
Prior to this work, the necessary conditions for ${\le}3$-confusability was not known and 
hence, methods to compute $T(n,\br)$ were not available.
\item For a fixed length and root, even though the set of descendants is huge, the set of all labels are significantly smaller. Hence, the task of computing maximum cliques remains feasible. More concretely, when $n=20$,
the order of the largest graph ${\cal G}(20,\br)$ is {366}, 
despite the fact that the average size%
\footnote{The number of ${\le}3$-irreducible ternary words of length at most twenty is 27687.}
of $\Sigma_3^{20}\cap D^*_{\le 3}(\br)$ is $3^{20}/27687\approx 125935$.

\end{itemize}
\end{remark}

Finally, we develop recursive constructions in the following proposition,
whose proof is deferred to Appendix \ref{app:codes}.

\begin{proposition}\label{prop:recursive}
Let $\br=r_1r_2\cdots r_i\in{\rm Irr}_{\le 3}(i,3)$. Then the following holds.
\[
\small
T(n,\br)\ge
\begin{cases}
T(n-1,\br\setminus r_1), &
\mbox{if $r_1= r_3$},\\

\max\{2T(n-4,\br\setminus r_1), 3T(n-8,\br\setminus r_1)\}, & 
\mbox{if $r_1\ne r_3$, $r_1\ne r_4$},\\

\max\{2T(n-5,\br\setminus r_1r_2), 3T(n-10,\br\setminus r_1r_2)\}, &
\mbox{if $r_1\ne r_3$, $r_1=r_4$, $r_2\ne r_5$},\\

\max\{2T(n-6,\br\setminus r_1r_2r_3), 3T(n-12,\br\setminus r_1r_2r_3)\}, &
\mbox{if $r_1\ne r_3$, $r_1=r_4$, $r_2= r_5$}.\\
\end{cases}
\]
Furthermore, $T(n,\br)\ge T(n-1,\br)$ and $T(n,\br)=T(n,\br^R)$,
where $\bz^R$ denotes the reverse of word $\bz$. 
\end{proposition}

{Using Proposition \ref{prop:recursive} with 
Constructions \ref{code:2} and \ref{code:oneregion}
and the values computed by {\tt MaxCliqueDyn},
we derive lower bounds for $T(n)$ for {$21\le n\le 30$}.
The results are summarized in Table \ref{table:tn}.
In addition to the lower bounds for the code size $T(n)$, 
we also compare the upper bounds in Proposition \ref{prop:irr2-upper} and \eqref{eq:upper}.
Observe that \eqref{eq:upper} is tight up to lengths at most ten and 
the constructions in this paper improve the rates%
\footnote{The rate of a code $\C$ of length $n$ is given by $(\log_3 |\C|)/n$.}
 for Construction \ref{code:jain} by as much as {6.74\%}.

\begin{table}
\caption{Estimates and Exact Values for $T(n)$}
\label{table:tn}
\centering
\begin{tabular}{|c|r|r|r|r||c|r|r|r|r|}
\hline
Length &
{Constr. \ref{code:jain}} &
This paper &
Eq \eqref{eq:upper} &
Prop. \ref{prop:irr2-upper}&
Length &
{Constr. \ref{code:jain}} &
This paper &
Eq \eqref{eq:upper} &
Prop. \ref{prop:irr2-upper} \\
\hline
1 & {\bf 3} & {\bf 3} & {\bf 3} & {\bf 3} & 
 16 & 5985 & {\bf 10641} & 12267 & 15495 \\
 2 & {\bf 9} & {\bf 9} & {\bf 9} & {\bf 9} & 
 17 & 8781 & {\bf 16287} & 19479 & 25077 \\
 3 & {\bf 21} & {\bf 21} & {\bf 21} & {\bf 21} & 
 18 & 12879 & {\bf 25005} & 30957 & 40581 \\
 4 & {\bf 39} & {\bf 39} & {\bf 39} & {\bf 39} & 
 19 & 18885 & {\bf 38223} & 49245 & 65667 \\
 5 & {\bf 69} & {\bf 69} & {\bf 69} & {\bf 69} & 
 20 & 27687 & {\bf 57957} & 78417 & 106257 \\
 6 & 111 & {\bf 117} & {\bf 117} & 117 & 
21 & 40587 & 83619 & 125001 & 171933 \\
 7 & 171 & {\bf 195} & {\bf 195} & 195 & 
22 & 59493 & 116145 & 199467 & 278199 \\
 8 & 261 & {\bf 315} & {\bf 315} & 321 & 
23 & 87201 & 166761 & 318621 & 450141 \\
 9 & 393 & {\bf 495} & {\bf 495} & 525 & 
24 & 127809 & 249159 & 509457 & 728349 \\
 10 & 585 & {\bf 777} & {\bf 777} & 855 & 
25 & 187323 & 375129 & 815361 & 1178499 \\
 11 & 867 & {\bf 1221} & 1227 & 1389 & 
26 & 274545 & 558573 & 1306107 & 1906857 \\
 12 & 1281 & {\bf 1887} & 1941 & 2253 & 
27 & 402375 & 813771 & 2093967 & 3085365 \\
 13 & 1887 & {\bf 2913} & 3075 & 3651 & 
28 & 589719 & 1164309 & 3359685 & 4992231 \\
 14 & 2775 & {\bf 4527} & 4875 & 5913 & 
29 & 864285 & 1675935 & 5394369 & 8077605 \\
 15 & 4077 & {\bf 6969} & 7731 & 9573 & 
30 & 1266681 & 2464419 & 8667075 & 13069845 \\
 
\hline
\end{tabular}
%
%
%
%
%
%
%
%
%
%
%
%
%
%
%
%
%

\vspace{1mm}
Optimal values of $T(n)$ are highlighted in {\bf bold}.
\vspace{-2mm}
\end{table}

\section{Discussion}

We studied the problem of determining the ${\le}k$-confusability of two words.
Combining the results of this paper, we have linear-time algorithms to solve the confusability problem for $k\in\{1,2,3\}$.

It remains open whether there exist efficient algorithms to determine ${\le}4$-confusability. 
One key obstacle is the fact that Proposition \ref{prop:jain2}(iii) does not hold for $k=4$.
In particular, there exists $\bx$ and $\by$ such that $R_{\le 4}(\bx)\ne R_{\le 4}(\by)$, 
but $\bx$ and $\by$ are ${\le}4$-confusable.
An example is provided by Jain {\em et al.}\cite{Jain}, 
where $\bx=012$ and $\by=0121012$ belongs to ${\rm Irr}_{\le4}(3)$, but
$012101212$ is a common descendant of $\bx$ and $\by$.

Another approach is to consider the intersection of the descendant cones 
$D_{\le 4}^*(\bx)\cap D_{\le 4}(\by)$ as formal languages.
Unfortunately, Leupold {\em et al.}\cite{Leupold1} demonstrated that 
while the language $D_{\le 4}^*(\bx)$ is context free for all $\bx$, the language is not regular in general.
However, given two context free languages $L_1$ and $L_2$, 
determining whether $L_1\cap L_2$ is empty is an undecidable problem
(see, for example, Sipser \cite[Exercise 5.32]{Sipser2006}). While this does not imply that the confusability problem is undecidable, 
we nevertheless conjecture that it is undecidable for $k\ge 4$.

\section*{Acknowledgement}

We are grateful for an anonymous reviewer who proposed the definition of label given in Section \ref{sec:codes}.
This definition has led to a more succint presentation and a spectrum of new results, including Proposition \ref{prop:upper} 
and the computation of the size of optimal codes of lengths up to twenty.

\bibliographystyle{IEEEtran}
\bibliography{mybibliography}

\begin{thebibliography}{10}
\providecommand{\url}[1]{#1}
\csname url@samestyle\endcsname
\providecommand{\newblock}{\relax}
\providecommand{\bibinfo}[2]{#2}
\providecommand{\BIBentrySTDinterwordspacing}{\spaceskip=0pt\relax}
\providecommand{\BIBentryALTinterwordstretchfactor}{4}
\providecommand{\BIBentryALTinterwordspacing}{\spaceskip=\fontdimen2\font plus
\BIBentryALTinterwordstretchfactor\fontdimen3\font minus
  \fontdimen4\font\relax}
\providecommand{\BIBforeignlanguage}[2]{{%
\expandafter\ifx\csname l@#1\endcsname\relax
\typeout{** WARNING: IEEEtran.bst: No hyphenation pattern has been}%
\typeout{** loaded for the language `#1'. Using the pattern for}%
\typeout{** the default language instead.}%
\else
\language=\csname l@#1\endcsname
\fi
#2}}
\providecommand{\BIBdecl}{\relax}
\BIBdecl

\bibitem{Lander}
E.~S. Lander, L.~M. Linton, B.~Birren, C.~Nusbaum, M.~C. Zody, J.~Baldwin,
  K.~Devon, K.~Dewar, M.~Doyle, W.~FitzHugh \emph{et~al.}, ``Initial sequencing
  and analysis of the human genome,'' \emph{Nature}, vol. 409, no. 6822, pp.
  860--921, 2001.

\bibitem{Mundy}
N.~I. Mundy and A.~J. Helbig, ``Origin and evolution of tandem repeats in the
  mitochondrial {DNA} control region of shrikes ({L}anius spp.),''
  \emph{Journal of Molecular Evolution}, vol.~59, no.~2, pp. 250--257, 2004.

\bibitem{Usdin}
K.~Usdin, ``The biological effects of simple tandem repeats: lessons from the
  repeat expansion diseases,'' \emph{Genome research}, vol.~18, no.~7, pp.
  1011--1019, 2008.

\bibitem{Sutherland}
G.~R. Sutherland and R.~I. Richards, ``Simple tandem {DNA} repeats and human
  genetic disease.'' \emph{Proceedings of the National Academy of Sciences},
  vol.~92, no.~9, pp. 3636--3641, 1995.

\bibitem{Fondon}
J.~W. Fondon and H.~R. Garner, ``Molecular origins of rapid and continuous
  morphological evolution,'' \emph{Proceedings of the National Academy of
  Sciences}, vol. 101, no.~52, pp. 18\,058--18\,063, 2004.

\bibitem{Leupold1}
P.~Leupold, V.~Mitrana, and J.~M. Sempere, ``Formal languages arising from gene
  repeated duplication,'' in \emph{Aspects of Molecular Computing}.\hskip 1em
  plus 0.5em minus 0.4em\relax Springer, 2003, pp. 297--308.

\bibitem{Leupold2}
P.~Leupold, C.~Martin-Vide, and V.~Mitrana, ``Uniformly bounded duplication
  languages,'' \emph{Discrete Applied Mathematics}, vol. 146, no.~3, pp.
  301--310, 2005.

\bibitem{Jain2}
S.~Jain, F.~Farnoud, and J.~Bruck, ``Capacity and expressiveness of genomic
  tandem duplication,'' in \emph{Information Theory (ISIT), 2015 IEEE
  International Symposium on}.\hskip 1em plus 0.5em minus 0.4em\relax IEEE,
  2015, pp. 1946--1950.

\bibitem{Farnoud}
F.~Farnoud, M.~Schwartz, and J.~Bruck, ``The capacity of string-duplication
  systems,'' \emph{IEEE Transactions on Information Theory}, vol.~62, no.~2,
  pp. 811--824, 2016.

\bibitem{Arita}
M.~Arita and Y.~Ohashi, ``Secret signatures inside genomic {DNA},''
  \emph{Biotechnology progress}, vol.~20, no.~5, pp. 1605--1607, 2004.

\bibitem{Heider}
D.~Heider and A.~Barnekow, ``{DNA}-based watermarks using the {DNA}-{C}rypt
  algorithm,'' \emph{BMC Bioinformatics}, vol.~8, no.~1, p. 176, 2007.

\bibitem{Liss}
M.~Liss, D.~Daubert, K.~Brunner, K.~Kliche, U.~Hammes, A.~Leiherer, and
  R.~Wagner, ``Embedding permanent watermarks in synthetic genes,'' \emph{PloS
  one}, vol.~7, no.~8, p. e42465, 2012.

\bibitem{Jain}
S.~Jain, F.~Farnoud, M.~Schwartz, and J.~Bruck, ``Duplication-correcting codes
  for data storage in the {DNA} of living organisms,'' \emph{IEEE Transactions
  on Information Theory}, vol.~63, no.~8, pp. 4996--5010, 2017.

\bibitem{Navarro:2001}
G.~Navarro, ``A guided tour to approximate string matching,'' \emph{ACM
  computing surveys}, vol.~33, no.~1, pp. 31--88, 2001.

\bibitem{Konc}
\BIBentryALTinterwordspacing
J.~Konc and D.~Janezic, ``An improved branch and bound algorithm for the
  maximum clique problem,'' \emph{proteins}, vol.~4, no.~5, 2007. [Online].
  Available: \url{http://insilab.org/maxclique/}
\BIBentrySTDinterwordspacing

\bibitem{Sipser2006}
M.~Sipser, \emph{Introduction to the Theory of Computation}.\hskip 1em plus
  0.5em minus 0.4em\relax Thomson Course Technology Boston, 2006, vol.~2.

\bibitem{Heubach:2009}
S.~Heubach and T.~Mansour, \emph{Combinatorics of compositions and
  words}.\hskip 1em plus 0.5em minus 0.4em\relax CRC Press, 2009.

\end{thebibliography}

\appendices

\section{Computing the Root in Linear Time}\label{app:root}

We formally describe linear-time algorithms to compute all ${\le}k$-roots for a word for $k\in\{2,3\}$.
Given any $\bx\in \Sigma_q^*$, 
since the root set of $\bx$ has size exactly one by Proposition \ref{prop:jain1}, 
it suffices to find one $\br\in{\rm Irr}_{\le k}(q)$ so that 
$\br\xRightarrow[\le k]{*}\bx$.
Lemma \ref{lem:ordering} then implies that we may reorder the duplications and 
assume that the tandem duplications are performed in decreasing lengths.

Therefore, to compute $\br$, we simply remove duplicates in increasingly length.
Formally, given a word of the form $\bu\bv\bv\bw$ where $|\bv|=k'$, we say that $\bv$ is a {\em $k'$-duplicate}.
{\em Removing the $k'$-duplicate} $\bv$ yields the word $\bu\bv\bw$.
Therefore, to compute ${\le}2$-root, we simply remove all 1-duplicates from $\bx$ and 
then all 2-duplicates from the result.
If we want to compute ${\le}3$-root, we remove all 3-duplicates from the ${\le}2$-root.

A formal description of the algorithm in Algorithm \ref{alg:root}.
Since removing all $k'$-duplicates may be performed in linear time, 
the algorithms run in linear-time.

\begin{algorithm}[h!]
\caption{Finding the ${\le}k$-root}
\label{alg:root}
\begin{algorithmic}[1]
\small
\Procedure{Find-${\le}2$-Root}{$\bx$}
\Require $\bx\in\Sigma_q^*$
\Ensure ${\br}$, where $R_{\le 2}(\bx)=\{\br\}$. 
\State $\br\gets \bx$
\State Remove all 1-duplicates in $\br$ using {\sc RemoveDuplicates}$(\br,1)$
\State Remove all 2-duplicates in $\br$ using {\sc RemoveDuplicates}$(\br,2)$
\EndProcedure
\Statex
\Procedure{Find-${\le}3$-Root}{$\bx$}
\Require $\bx\in\Sigma_q^*$
\Ensure ${\br}$, where $R_{\le 2}(\bx)=\{\br\}$. 
\State $\br\gets \bx$
\State Remove all 1-duplicates in $\br$ using {\sc RemoveDuplicates}$(\br,1)$
\State Remove all 2-duplicates in $\br$ using {\sc RemoveDuplicates}$(\br,2)$
\State Remove all 3-duplicates in $\br$ using {\sc RemoveDuplicates}$(\br,3)$
\EndProcedure
\Statex

\Procedure {RemoveDuplicates}{$\br$,$k$}
 \While{$i+2k-1 \le |\br|$} 
 \If{the substring of length $k$ starting at index $i$ is equal to the substring of length $k$ starting at index $i+k$}
  \State{Remove the substring of length $k$ starting at index $i$}
\Else
\State{$i\gets i+1$} 
\EndIf
 \EndWhile
\EndProcedure
\end{algorithmic}
\end{algorithm}

\section{Proofs of Lemmas \ref{lem:reorder} and \ref{lem:distinct}}\label{app:order}

\noindent{\bf Lemma 5.}
Let $k_1<k_2\le 3$ and $\bx\in \Sigma_q^*$.
Suppose that $T_{i_2,k_2}\circ T_{i_1,k_1}(\bx)={\bx'}$ for some $i_1,i_2$.
Then there exist integers $3\ge \ell_1\ge \ell_2\ge \cdots\ge \ell_t$ and $j_1,j_2,\ldots, j_t$
such that
\[
     T_{j_t,\ell_t}\circ T_{j_{t-1},\ell_{t-1}}\circ \cdots \circ T_{j_1,\ell_1}({\bx})=\bx'. 
 \]

\begin{proof}
Consider a sequence of two tandem duplications
$T_{i_2,k_2}\circ T_{i_1,k_1}$ with $k_1< k_2\le 3$.
We replace the sequence of duplications according to the following rules.

    \begin{itemize}
        \item Suppose that $k_1=1$ and $k_{2}=3$.
        \begin{enumerate}[(a)]
            \item If $i_{2}\le i_1 - 2$, then substitute with 
            $ T_{i_{2},3}\circ T_{i_1+3,1}$.
            \item If $i_{2}=i_1 -1$, then substitute with 
            $T_{i_1-1,2}\circ T_{i_{1},1}\circ T_{i_1+3,1}$.
            \item If $i_{2}= i_1$, then substitute with 
            $T_{i_1,2}\circ T_{i_{1},1}\circ T_{i_1+3,1}$.
            \item If $i_{2}\ge i_1 +1$, then substitute with 
            $T_{i_2-1,3} \circ T_{i_{1},1}$.
        \end{enumerate}
        \item Suppose that $k_1=2$ and $k_{2}=3$.
        \begin{enumerate}[(a)]
            \item If $i_{2}\le i_1 - 1$, then substitute with 
            $T_{i_2,3} \circ T_{i_{1}+3,2}$.
            \item If $i_{2}=i_1$, then substitute with 
            $T_{i_{1},2} \circ T_{i_1,2} \circ T_{i_{2},1} $.
            \item If $i_{2}= i_1+1$, then substitute with 
            $T_{i_{1},2}\circ T_{i_1,2}\circ T_{i_1+3,1}$
            \item If $i_{2}\ge i_1 +2$, then substitute with 
            $ T_{i_{2}-2,3}\circ T_{i_1,2}$.
        \end{enumerate}
        \item Suppose that $k_1=1$ and $k_{2}=2$.
        \begin{enumerate}[(a)]
            \item If $i_{2}\le i_1 - 1$, then substitute with 
            $T_{i_{2},2}\circ T_{i_1+2,1}$.
            \item If $i_{2}=i_1$, then substitute with 
            $T_{i_{1},1} \circ T_{i_{1},1} \circ T_{i_{1},1}$.
            \item If $i_{2}\ge i_1+1$, then substitute with 
            $T_{i_{2}-1,2}\circ T_{i_{1},1}$
        \end{enumerate}
    \end{itemize}
Thus, given a sequence of tandem duplications, we may reorder the tandem duplications such that the lengths of the repeats are non-increasing. 
\end{proof}

\noindent{\bf Lemma 6.} Let $k\in \{2,3\}$. Suppose $T_{i,k}(\bx)=\bu\bv\bv\bw$, where $\bx=\bu\bv\bw$, $|\bu|=i$ and $|\bv|=k$.
If the symbols in $\bv$ are not pairwise distinct, then 
there exist integers $k> \ell_1\ge \ell_2$ and $j_1,j_2$
such that
\[
     T_{j_2,\ell_2}\circ T_{j_1,\ell_1}({\bx})=T_{i,k}(\bx). 
 \]

\begin{proof}
When $k=3$, we consider the following three cases.
\begin{itemize}
\item Duplication of $aba$ to $abaaba$ is equivalent to
$T_{2,1}\circ T_{1,2}(aba)=T_{2,1}(ababa)=abaaba$.
\item Duplication of  $aab$ to $aabaab$ is equivalent to
$T_{3,1}\circ T_{1,2}(aab)=T_{3,1}(aabab)=aabaab$.
\item Duplication of  $abb$ to $abbabb$ is equivalent to
$T_{1,1}\circ T_{0,2}(abb)=T_{1,1}(ababb)=abbaab$.
\item Duplication of $aaa$ to $aaaaaa$ is equivalent to
$T_{0,1}\circ T_{0,2}(aaa)=T_{0,1}(aaaaa)=aaaaaa$.
\end{itemize}

When $k=2$, 
we consider the duplication of length two on $aa$ to obtain $aaaa$.
This is equivalent to $T_{0,1}\circ T_{0,1}(aa)=T_{0,1}(aaa)=aaaa$. 
Therefore, if a duplication is performed on a segment whose symbols are not distinct, we may find an equivalent sequence of duplications of strictly shorter lengths.
\end{proof}

\section{Proof of Corollary \ref{cor:main}}\label{app:cor1}

\noindent{\bf Corollary 12.}
Suppose that ${\bx}$ and ${\by}$ are two words such that ${\br}=R_{\le 3}({\bx})=R_{\le3}({\by})$ and
		$\br$ contains at least three distinct symbols.
		Set 
		$\bp={\rm Ext}({\rm Reg}({\br}),{\bx})$ and $\bq={\rm Ext}({\rm Reg}({\br}),{\by})$.
		Suppose further that 
		${\rm Count}({\rm main}({\br}),R_{\le2}(\bp))<{\rm Count}({\rm main}({\br}),R_{\le 2}(\bq))$.
		Then ${\bx}$ and ${\by}$ are confusable if and only if 
		${\rm Count}({\rm rot}({\rm main}({\br})),\bp)>0$ and ${\bx}\setminus{\rm *Pref}({\br},{\bx})$ and ${\by}\setminus{\rm *Pref}({\br},{\by})$ are confusable.

\begin{proof}
By Theorem \ref{thm:main}, it is sufficient to show that ${\rm Count}({\rm rot}({\rm main}({\br})),\bp)>0$ iff $\bp$ and $\bq$ are confusable. Let $R_{\le 2}(\bp)=\bw(abc)^k ab$ and $R_{\le 2}(\bq)=\bw(abc)^j ab$ for some $k < j$.
	
	Suppose there exists at least one distinct triplet in $\bp$, we set ${\bp'}$ to be the result of duplication of the distinct triplet in $\bp$ as many as $j-k$ times and ${\bq'}$ to be ${\bq}$. Then $R_{\le 2}(\bp')=\bw(abc)^j ab=R_{\le 2}(\bq')$, and hence by Theorem~\ref{confusable}, $\bp$ and $\bq$ are confusable.
	
	Suppose on the other hand, $\bp$ and $\bq$ are confusable. By Theorem~\ref{confusable}, there exist ${\bp'}$ and ${\bq'}$ such that $\bp \xRightarrow[3_d]{*} \bp'$ and $\bq \xRightarrow[3_d]{*} \bq'$, where $R_{\le 2}(\bp')=R_{\le 2}(\bq')=\bw(abc)^g ab$ for some $g \ge j >k.$ Then there has to be at least one substring of length three with distinct symbols, or {\em distinct triplet}, in $\bp$. 
	Since $\bp={\rm Ext}({\rm Reg}({\br}),{\bx})$, then by definition, 
	the only distinct triplet possible is from the set ${\rm rot}({\rm main}({\bf r}))$.
	Therefore, ${\rm Count}({\rm rot}({\rm main}({\bf r})),\bp)>0$.
	\end{proof}

\section{On the Sum of Lengths of All Prefixes}\label{app:prefixes}

Recall in Section \ref{sec:implementation}, the running time of Algorithm \ref{alg:confuse}
is linear in the sum of the lengths of all prefixes that are computed in all recursive calls.
Therefore, if this sum is linear in the lengths of the original words, Algorithm \ref{alg:confuse} runs in linear time.
Specifically, we establish the following proposition.

\begin{proposition}
Let $\bx\in\Sigma_q^m$. Let $\bp_1,\bp_2,\ldots, \bp_s$ be the prefixes computed in Line 13 in Algorithm \ref{alg:confuse}. Then $\sum_{i=1}^s|\bp_i|\le 3m$.
\end{proposition}

\begin{proof}
Let $\bx=x_1x_2\cdots x_m$. 
To establish the proposition, it suffices to show that each symbol $x_j$ with $1\le j\le m$ appears%
\footnote{In this proof, when we refer to a symbol $x_j$, we refer specifically to the symbol at index $j$.
More formally, we may rewrite the word $\bx$ as $(x_1,1)(x_2,2)\cdots (x_m,m)$.}
in at most three prefixes $\bp_i, \bp_{i+1}, \bp_{i+2}$ for some $1\le i\le s-2$.

Suppose that $\bp_i$ is the first prefix where $x_j$ appears in and 
let $\ba_i$ be the suffix that is called in the next recursive call.
If $x_j$ is not removed, then $x_j$ must belong the overlapping region of $\bp_i$ and $\ba_i$
(see Fig. \ref{fig:prefixes}).
Suppose that $\br_i$ is the current root and ${\rm main}(\br_i)=abc$. 
Then the overlapping region is of the form $ab^\ell$ for some $\ell\ge 1$ and we have the following cases.
\begin{itemize}
\setlength\itemsep{2mm}
\item \underline{$x_j=a$}. Then $x_j$ is necessarily the first symbol of $\ba_i$ and is removed in the next recursive call.
In other words, $x_j$ appears only in two prefixes.

\item \underline{$x_j=b$ for some $b$ in the overlapping region}.
In the next recursive call or $(i+1)$th call, either $ab^{\ell-1}$ or $ab^\ell$ is removed.
If $x_j$ is not the last $b$ in the overlapping region, $x_j$ is removed and hence, appears only in two prefixes.
If $x_j$ is the last $b$, then $x_j$ is the first symbol in the $(i+2)$th call and
is removed in this call. Hence, $x_j$ appears in three prefixes in this case. \qedhere
\end{itemize}
\end{proof}

\section{Proof of Proposition \ref{prop:upper}}\label{app:upper}

Let $\bx$ be a word of length $n$ whose root is $\br$.
Suppose that $\br$ has length $i$ and $m$ regions and 
we set ${\rm Label}(\bx)=(\br,(c_1,\delta_1),(c_2,\delta_2),\ldots,(c_m,\delta_m))$.
To prove Proposition \ref{prop:upper}, we first show certain properties of the label of $\bx$.

\begin{lemma}We have that
\begin{equation}\label{numsol}
i+3(c_1+c_2+\cdots+c_m)-m\le n.
\end{equation}
\end{lemma}

\begin{proof}
Since $\br \xRightarrow[{\le}3]{*} \bx$, then by Theorem \ref{confusable}, there exists $\bx'$ such that
$\br \xRightarrow[3_d]{*}  \bx'$, and $\bx' = R_{{\le}2}(\bx)$. Let $1 \leq j \leq m$. 
Since there are $c_j$ distinct triplets in the $j$th region of $R_{{\le}2} (\bx)$, 
we have to duplicate the distinct triplet in the $j$th region of $r$ by $c_j -1$ times. 
Therefore we have $|\bx'| = |\br| + \sum_{j=1}^{m}{3 (c_j -1)} = i + 3(c_1 + c_2 + \cdots c_m) - m$. 
Note that since $\bx'=R_{{\le}2}(\bx)$, we have $|\bx| \geq |\bx'|$ and hence it yields \eqref{numsol}.
%
\end{proof}

\begin{lemma}We have that $\delta_1=\delta_2=\cdots=\delta_m=+$, whenever
\begin{equation}\label{allplus}
i+3(c_1+c_2+\cdots+c_m)-m=n.
\end{equation}
\end{lemma}

\begin{proof}
Since we know that $i + 3(c_1 + c_2 + \cdots c_m) - m$ is the length of $\bx'= R_{{\le}2} (\bx)$,  
and the equality holds, that means there is no tandem duplication of length at most two in $\bx$. 
 Hence, a distinct triplet must remain in each of the regions and 
 so, $\delta_j=+$ for all $1\le j\le m$.
\end{proof}

Finally, we complete the proof of Proposition \ref{prop:upper}.

\noindent{\bf Proposition 15.}
Let $i\le n$. Suppose that $\br\in {\rm Irr}_{{\le}3}(i,3)$ has $m$ regions. 
Then 
\[ T(n,\br)\le U(n,i,m)\triangleq
\begin{cases}
\binom{(n-i)/3+m}{m}-\binom{(n-i)/3+m-1}{m-1}+1, & \mbox{if 3 divides $n-i$},\\
\binom{\floor{(n-i)/3}+m}{m}, &\mbox{otherwise.}
\end{cases}
\]

\begin{proof}
Let $\C$ be an $(n\le 3;\br)$-TD code.
For $\bx\in \C$, set ${\rm Label}(\bx)=(\br,(c_1,\delta_1),(c_2,\delta_2),\ldots,(c_m,\delta_m))$
and so, $(c_1,c_2,\ldots, c_m)$ is an integer solution to \eqref{numsol} whose entries are all positive.
Corollary \ref{cor:label}(i) implies that these integer solutions are distinct for different $\bx$ chosen from $\C$.
Hence, the number of integer solutions to \eqref{numsol} is an upper bound to the size of $\C$.
This number is wellknown (see Heubach and Mansour \cite{Heubach:2009}) and is given by $\binom{\floor{(n-i)/3}+m}{m}$.

When $n-i$ is divisible by three, we improve this upper bound. 
Observe that the integer solutions to \eqref{allplus} is a proper subset of \eqref{numsol}.
Corollary \ref{cor:label}(ii) then implies that there is at most one $\bx$ in $\C$ whose label yields a positive integer solution to \eqref{allplus}. 
Since the number of positive integer solutions to \eqref{allplus} is $\binom{(n-i)/3+m-1}{m-1}$, we obtain the desired upper bound.
\end{proof}

\section{Proof of Code Constructions}\label{app:codes}

We provide proofs of Construction \ref{code:oneregion} and Proposition \ref{prop:recursive}.
\vspace{2mm}

\noindent\textbf{Construction \ref{code:oneregion}.}
For $\br\in R$ and $\ell\ge 1$, define the words $\bx(\br,\ell)$ and $\bz(\br,\ell)$ with the following rule.
\vspace{1mm}

\begin{center}
\begin{tabular}{|l|l|l||l|l|l|}
\hline
$\br$ & $\bx(\br,\ell)$ & $\bz(\br,\ell)$ &
$\br$ & $\bx(\br,\ell)$ & $\bz(\br,\ell)$ \\ \hline
$012$ & $0(112200)^{\ell-1} 112$ & $(012)^\ell$ &
$1012$ & $10(112200)^{\ell-1} 112$ & $1(012)^\ell$ \\ \hline
$0120$ & $0(112200)^{\ell-1} 11220$ & $(012)^\ell0$ &
$10120$ & $10(112200)^{\ell-1} 11220$ & $1(012)^\ell0$ \\ \hline
$01201$ & $0(112200)^{\ell-1} 1122001$ & $(012)^\ell01$ &
$101201$ & $10(112200)^{\ell-1} 1122001$ & $1(012)^\ell01$ \\ \hline
$0121$ & $0(112200)^{\ell-1} 1121$ & $(012)^\ell1$ &
$10121$ & $10(112200)^{\ell-1} 1121$ & $1(012)^\ell1$ \\ \hline
$01202$ & $0(112200)^{\ell-1} 112202$ & $(012)^\ell02$ &
$101202$ & $10(112200)^{\ell-1} 112202$ & $1(012)^\ell02$ \\ \hline
$012010$ & $0(112200)^{\ell-1} 11220010$ & $(012)^\ell010$ &
$1012010$ & $10(112200)^{\ell-1} 11220010$ & $1(012)^\ell010$ \\ \hline
\end{tabular}
\end{center}
%
For $\br\in R$ and $n\ge |\bx(\br,2)|$, set $\ell_z=\floor{(n-|\br|)/3}+1$ and
\[{\cal C}_{\br}(n)
\triangleq \left\{\xi_{n-|\bx(\br,\ell)|}(\bx(\br,\ell)):|\bx(\br,\ell)| \le n\right\}
\cup \left\{\xi_{n-|\bz(\br,\ell_z)|} (\bz(\br,\ell_z)) \right\}.
\]
Then ${\cal C}_{\br}(n)$ is an $(n,{\le}3;\br)$-TD code. 
Therefore, if we set $n_2=|\bx(r,2)|$,
we have that
\[ T(\br,n) = 
\begin{cases} 
\floor{\frac{n-n_2}{6}}+3 &\mbox{if $n\ge n_2$},\\
2, &\mbox{if $|\br|+3\le n< n_2$},\\
1, &\mbox{otherwise.}
\end{cases}
\]

\begin{proof}
Observe that for all $\ell\ge 1$, we have 
${\rm Label}(\bx(\br,\ell))=(\br,(\ell,-))$ and ${\rm Label}(\bz(\br,\ell))=(\br,(\ell,+))$.
Furthermore, the length of $\bx(\br,\ell)$ is at least $6(\ell-1)+|\br|+1$. 

Let ${\cal X}=\left\{\xi_{n-|\bx(\br,\ell)|}(\bx(\br,\ell)):|\bx(\br,\ell)| \le n\right\}$ and set 
$M=|{\cal X}|=\max\{\ell: \bx(\br,\ell))\le n \}$.
Hence, we have $6(M-1)+|\br|+1\le n$. 
So, $\ell_z=\floor{(n-|\br|)/3}+1\ge \floor{(6M-5)/3}+1\ge M+1$ for $M\ge 2$.

Hence, for any word $\bx$ in ${\cal X}$, 
we have that ${\rm Label}(\bx)=(\br,(\ell',-))$ with $\ell'\le M<\ell_z$.
Therefore, Proposition \ref{prop:label} implies that $\bx$ and $\bz(\br,\ell_z)$ are not confusable.
Also, Corollary \ref{cor:label}(ii) implies that $\bx$ and $\bx'$ are not confusable for $\bx, \bx'\in {\cal X}$.
Therefore, ${\cal C}_{\br}(n)$ is an $(n,{\le}3;\br)$-TD code. 

To demonstrate optimality, we observe that for $\by\in D_{{\le}3}^*(\br)$, if ${\rm Label}(\by)=(\ell,-)$, 
then $|\by|\ge |\bx(\br,\ell)|$. Similarly, if ${\rm Label}(\by)=(\ell,+)$, then $|\by|\ge |\bz(\br,\ell)|$.
Suppose that there is an $(n,{\le}3;\br)$-TD code ${\cal C'}$ with size $M+2$.
We know from Corollary \ref{cor:label}(ii) that there can be no two codewords whose label $\delta_1$ are $+$.
Furthermore, by Corollary \ref{cor:label}(i),
 there exists an $\by\in {\cal C'}$ whose label is $(c_1,-)$ with $c_1\ge M+1$.
This implies that $|\by|\ge |\bx(\br,M+1)|>n$, a contradiction.

Finally, the values of $T(n,\br)$ follow from straightforward computations.
\end{proof}

\noindent\textbf{Proposition \ref{prop:recursive}.}
Let $\br=r_1r_2\cdots r_i\in{\rm Irr}_{\le 3}(i,3)$. Then the following holds.
\[
T(n,\br)\ge
\begin{cases}
T(n-1,\br\setminus r_1), &
\mbox{if $r_1= r_3$},\\

\max\{2T(n-4,\br\setminus r_1), 3T(n-8,\br\setminus r_1)\}, & 
\mbox{if $r_1\ne r_3$, $r_1\ne r_4$},\\

\max\{2T(n-5,\br\setminus r_1r_2), 3T(n-10,\br\setminus r_1r_2)\}, &
\mbox{if $r_1\ne r_3$, $r_1=r_4$, $r_2\ne r_5$},\\

\max\{2T(n-6,\br\setminus r_1r_2r_3), 3T(n-12,\br\setminus r_1r_2r_3)\}, &
\mbox{if $r_1\ne r_3$, $r_1=r_4$, $r_2= r_5$}.\\
\end{cases}
\]
Furthermore, $T(n,\br)\ge T(n-1,\br)$ and $T(n,\br)=T(n,\br^R)$,
where $\bz^R$ denotes the reverse of word $\bz$. 

\begin{proof}
First consider $r_1=r_3$. 
Suppose that ${\cal D}$ is an $(n-1,{\le}3;\br \setminus r_1)$-TD code.
To construct a code of length $n$, we simply prepend the prefix $r_1$ to all words in ${\cal D}$.
For convenience, given a set of words ${\cal X}$ and a word $\bp$, we use $\bp{\cal X}$
to denote the set $\{\bp\bx: \bx\in \X\}$. When $\bp=p$ is of length one, we simply write $p\X$.
Using this notation, we set $\C=r_1\D$.
We then apply Proposition \ref{prop:label} and verify that 
${\cal C}$ is an $(n,{\le}3;\br)$-TD code whose size is given by $|{\cal D}|$.

Next, consider $r_1\ne r_3$ and $r_1\ne r_4$. Set $\br'=\br \setminus r_1$ and
suppose that $\D$ is an $(n-4,{\le}3;\br')$-TD code.
Let $\D_1=r_1r_2r_2r_2\D$ and $\D_2=r_1r_2r_3r_1\D$, and 
so, $\D_1\cup \D_2\subseteq D^*_{{\le}3}(\br)\cap \Sigma^n_3$.

For $\bx\in\D$, let ${\rm Label}(\bx)=(\br',(c_1,\delta_1),\ldots,(c_m,\delta_m))$.
Then we have that 
\begin{align*}
{\rm Label}(r_1r_2r_2r_2\bx) &=(\br,(1,-),(c_1,\delta_1),\ldots,(c_m,\delta_m)),\\
{\rm Label}(r_1r_2r_3r_1\bx) &=(\br,(2,+),(c_1,\delta_1),\ldots,(c_m,\delta_m)).
\end{align*}
Proposition \ref{prop:label} implies that $r_1r_2r_2r_2\bx$ and $r_1r_2r_3r_1\bx$ are not confusable.
We can similarly check that any pair of distinct words in $\D_1\cup \D_2$ are not confusable.

For the remaining cases, we choose the short code ${\cal D}$ and prepend $\D$ according to the rules below.

\begin{center}
\begin{tabular}{|l|l|l|}
\hline
Conditions for $\br$ & Short Code ${\cal D}$& $(n,{\le}3;\br)$-TD code \\\hline

$r_1= r_3$ &
$(n-1,{\le}3;\br\setminus r_1)$-TD code &
$r_1$\\ \hline

$r_1\ne r_3, r_1\ne r_4$ &
$(n-4,{\le}3;\br\setminus r_1)$-TD code &
\hspace{3mm}$r_1r_2r_2r_2\D$\\
&& $\cup~r_1r_2r_3r_1\D$\\ \hline

$r_1\ne r_3, r_1\ne r_4$ &
$(n-8,{\le}3;\br\setminus r_1)$-TD code &
\hspace{3mm}$r_1r_2r_2r_2r_2r_2r_2r_2\D$\\
&& $\cup~r_1r_2r_2r_3r_3r_1r_1r_2\D$\\ 
&& $\cup~r_1r_2r_3r_1r_2r_3r_1r_2\D$\\ 
\hline

$r_1\ne r_3, r_1= r_4, r_2\ne r_5$ &
$(n-5,{\le}3;\br\setminus r_1r_2)$-TD code &
\hspace{3mm}$r_1r_2r_2r_2r_3\D$\\
&& $\cup~r_1r_2r_3r_1r_2\D$\\ \hline

$r_1\ne r_3, r_1= r_4, r_2\ne r_5$ &
$(n-10,{\le}3;\br\setminus r_1r_2)$-TD code &
\hspace{3mm}$r_1r_2r_2r_3r_3r_3r_3r_3r_3r_3\D$\\
&& $\cup~r_1r_2r_2r_3r_3r_1r_1r_2r_2r_3\D$\\ 
&& $\cup~r_1r_2r_3r_1r_2r_3r_1r_2r_3r_3\D$\\ 
\hline

$r_1\ne r_3, r_1= r_4, r_2= r_5$ &
$(n-6,{\le}3;\br\setminus r_1r_2r_3)$-TD code &
\hspace{3mm}$r_1r_2r_2r_3r_3r_1\D$\\
&& $\cup~r_1r_2r_3r_1r_2r_3\D$\\ \hline

$r_1\ne r_3, r_1= r_4, r_2= r_5$ &
$(n-12,{\le}3;\br\setminus r_1r_2r_3)$-TD code &
\hspace{3mm}$r_1r_2r_2r_3r_3r_1r_1r_1r_1r_1r_1r_1\D$\\
&& $\cup~r_1r_2r_2r_3r_3r_1r_1r_2r_2r_3r_3r_1\D$\\ 
&& $\cup~r_1r_2r_3r_1r_2r_3r_1r_2r_3r_1r_1r_1\D$\\ 
\hline

\end{tabular}
\end{center}

To show that $T(n,\br)\ge T(n-1,\br)$, let ${\cal D}$ be an $(n-1,{\le}3;\br)$-TD code. Then ${\cal C}=\{\xi_1(\bx):\bx\in {\cal D}\}$ is an $(n,{\le}3;\br)$-TD code.

To show that $T(n,\br)= T(n,\br^R)$, let ${\cal D}$ be an $(n,{\le}3;\br)$-TD code. Then ${\cal C}=\{\bx^R:\bx\in {\cal D}\}$ is an $(n,{\le}3;\br^R)$-TD code.
\end{proof}

\section{Number of Irreducible Words with Certain Number of Regions}

Recall that $I(i,m)$ denote the number of irreducible words  in ${\rm Irr}_{{\le}3}(i,3)$ with exactly $m$ regions. 
In this appendix, we derive a recursive formula for $I(i,m)$ that allows us to efficiently compute \eqref{eq:upper}.

Let ${\rm Irr}(aba, i,m)$ and ${\rm Irr}(abc, i,m)$ denote the set of irreducible words of length $i$ with exactly $m$ regions
that have two and three distinct symbols, respectively, in their prefixes of length three. 
Let $I(aba,i,m)$ and $I(abc,i,m)$ denote the sizes of ${\rm Irr}(aba, i,m)$ and ${\rm Irr}(abc, i,m)$, respectively.
Without loss of generality, given $\bx$ in ${\rm Irr}(aba, i,m)$ or ${\rm Irr}(abc, i,m)$, 
we assume that 


We consider the maps,
\begin{align*}
\Phi_1: & {\rm Irr}(aba, i,m) \to {\rm Irr}(abc, i-1,m),\\
\Phi_2: & {\rm Irr}(abc, i-1,m)\to {\rm Irr}(aba, i,m) ,\\
\Psi_1: & {\rm Irr}(abc, i,m) \to {\rm Irr}(aba, i-1,m-1)\cup {\rm Irr}(aba, i-2,m-1)\cup {\rm Irr}(aba, i-3,m-1),\\
\Psi_2: & {\rm Irr}(aba, i-1,m-1)\cup {\rm Irr}(aba, i-2,m-1)\cup {\rm Irr}(aba, i-3,m-1)\to {\rm Irr}(abc, i,m),
\end{align*}
defined via the following rules. In what follows, we set $\br=r_1r_2\cdots r_{|\br|}$ and for distinct elements $r_1,r_2$, set $r^*$ to be the unique symbol distinct from $r_1$ and $r_2$.
\begin{align*}
\Phi_1(\br) & = \br\setminus r_1,&
\Phi_2(\br) & = r_2\br,\\
\Psi_1(\br)&=
\begin{cases}
\br\setminus r_1, & \mbox{if $r_1\ne r_4$},\\
\br\setminus r_1r_2, & \mbox{if $r_1=r_4, r_2\ne r_5$},\\
\br\setminus r_1r_2r_3, & \mbox{if $r_1=r_4, r_2=r_5$},
\end{cases}
&
\Psi_2(\br)&=
\begin{cases}
r^*\br, & \mbox{if $|\br|=i-1$},\\
r_2r^*\br, & \mbox{if $|\br|=i-2$},\\
r_1r_2r^*\br, & \mbox{if $|\br|=i-3$}.
\end{cases}
\end{align*} 

Then we check that the maps $\Phi_1$, $\Phi_2$, $\Psi_1,$ and $\Psi_2$ are well-defined
and $\Phi_1\circ\Phi_2$, $\Phi_2\circ\Phi_1$, $\Phi_1\circ\Phi_2$, and $\Phi_2\circ\Phi_1$ are 
identity maps on their respective domains. Therefore, all four maps are bijections and we establish the following recursion.
For $m\ge 0$ and $i\ge 3$,

{ 
\begin{align*}
I(aba,i,m)&=
\begin{cases}
I(abc,i-1,m), &\mbox{$i>3$},\\
0, &\mbox{if $i=3, m>0$},\\
6, &\mbox{if $i=3, m=0$}.
\end{cases}\\
I(abc,i,m)&=
\begin{cases}
I(aba,i-1,m-1)+I(aba,i-2,m-1)+I(aba,i-3,m-1), &\mbox{if $i\ge 6,m\ge 1$},\\
0, &\mbox{if $i=5$, $m>2$},\\
6, &\mbox{if $i=5$, $m=2$},\\
12, &\mbox{if $i=5$, $m=1$},\\
0, &\mbox{if $i=4$, $m>1$},\\
12, &\mbox{if $i=4$, $m=1$},\\
0, &\mbox{if $i=3$, $m>1$},\\
6, &\mbox{if $i=3$, $m= 1$},\\
0, &\mbox{if $m= 0$},\\
\end{cases}
\end{align*}
}

Finally, to compute $I(i,m)$, we have that $I(i,m)=I(aba,i,m)+I(abc,i,m)$.

\end{document}